\documentclass[a4paper,11pt,DIV=11,% decrease borders by 2 (std 10 for 11pt)
abstract=on% return to old style centered abstract
]{scrartcl}
\usepackage{mathtools}
\mathtoolsset{showonlyrefs}
\usepackage{amsmath, amsthm, amssymb,accents}
\usepackage{fontenc}
\usepackage[utf8]{inputenc}
\usepackage[english]{babel}
\usepackage{graphicx}
\usepackage{subcaption,xargs}
\usepackage{algorithm, booktabs}
\usepackage[noend]{algpseudocode}
\usepackage{enumitem,listings,microtype}
\usepackage{hyperref}
\usepackage{environ,ifthen}
\usepackage{multirow}
\setlist[itemize]{label={--}}
\DeclareCaptionLabelSeparator{periodspace}{.\ }
\DeclareCaptionLabelSeparator{periodspace}{.\ }
\newtheorem{theorem}{Theorem}[section]

\newtheorem{lemma}[theorem]{Lemma}
\newtheorem{remark}[theorem]{Remark}
\newtheorem{example}[theorem]{Example}
\newtheorem{corollary}[theorem]{Corollary}
\setkomafont{sectioning}{\rmfamily\bfseries}
\setkomafont{title}{\rmfamily}
\setkomafont{descriptionlabel}{\rmfamily\bfseries}

\newcommand{\RR}{\mathbb{R}}

\newcommand{\SPD}{\mathcal P}

\makeatletter
\newcommandx{\abs}[2][1=\@empty]{#1\lvert #2 #1\rvert}
\newcommandx{\norm}[3][1=\@empty,3=\@empty]{#1\lVert #2 #1\rVert_{#3}}
\makeatletter

\newcommandx{\ubar}[1]{\underaccent{\bar}{#1}}

\newcommand{\tT}{\mathrm{T}}
\DeclareMathOperator*{\argmin}{arg\,min}
\DeclareMathOperator*{\argmax}{arg\,max}

\DeclareMathOperator{\sgn}{sgn}

\DeclareMathOperator{\Log}{log}
\DeclareMathOperator{\MLog}{Log}
\DeclareMathOperator{\Exp}{exp}
\DeclareMathOperator{\MExp}{Exp}
\DeclareMathOperator{\Diag}{diag}

\DeclareMathOperator{\KL}{KL}

\DeclareMathOperator{\distm}{dist}

\newcommand{\zb}[1]{\ensuremath{\boldsymbol{#1}}}

\hypersetup{pdfauthor={Ronny Bergmann, Jan Henrik Fitschen, Johannes Persch, Gabriele Steidl},
	pdftitle={Iterative Multiplicative Filters for Data Labeling},
	pdfsubject={submitted revised manuscript},%
	pdfcreator = {pdflatex and TextMate},%
	pdfkeywords={},
	plainpages=false, pdfstartview=FitH, pdfview=FitH, pdfpagemode=UseOutlines,%
	bookmarksnumbered=true,bookmarksopen=false,bookmarksopenlevel=0,%
	colorlinks=true,linkcolor=black,citecolor=black,urlcolor=black%
}
\begin{document}
\title{Iterative Multiplicative Filters for Data Labeling}
\date{\today}
\author{ 
Ronny Bergmann\footnotemark[1], Jan Henrik Fitschen\footnotemark[1], Johannes Persch\footnotemark[1], \\ and Gabriele Steidl\footnote{Department of Mathematics,
Technische Universität Kaiserslautern, Postfach 3049, 67653 Kaiserslautern, Germany,
\{bergmann, fitschen, persch, steidl\}@mathematik.uni-kl.de.
}
}
\maketitle
\begin{abstract}
	\noindent\small 
	 Based on an idea in~\cite{APSS2016} we propose a new iterative
   multiplicative filtering algorithm for label assignment matrices
   which can be used for the supervised partitioning of data.
   Starting with a row-normalized matrix containing the averaged distances
   between prior features and the observed ones the method assigns in a very
   efficient way labels to the data. 
   We interpret the algorithm as a gradient ascent method with respect
   to a certain function on the product manifold of positive numbers followed
   by a reprojection 	onto a subset of the probability simplex consisting of
   vectors whose components are bounded away from zero by a small constant.
   While such boundedness away from zero is necessary to avoid an arithmetic
   underflow, our convergence results imply that they are also necessary for
   theoretical reasons.
   Numerical examples show that the proposed simple and fast algorithm leads to 
   very good results. 
   In particular we apply the method for the partitioning 
   of manifold-valued images.	
\end{abstract}
%-------------------------------------------------------------------------------
\section{Introduction} \label{sec:intro}
%-------------------------------------------------------------------------------
Data labeling  is a basic problem which appears in various applications.
In particular it can be used for image partitioning and segmentation, 
which is an important preprocessing step for many state-of-the-art algorithms
in high-level computer vision. 
It addresses the task to assign labels from a finite set
to the image points in a meaningful way. 
Thereby the number of labels $K$ is much smaller than the image dimension $n$.
Fig.~\ref{fig:intro} illustrates a typical labeling result
which leads to the segmentation of a texture image.

\begin{figure}
	\centering
	\begin{subfigure}{0.32\textwidth}
		\centering
		\includegraphics[height=0.98\textwidth]{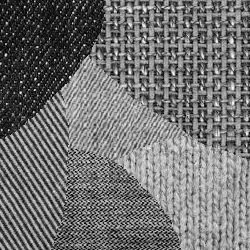}
		\caption[]{Original image}\label{fig:intro:orig}
	\end{subfigure}
	\begin{subfigure}{0.32\textwidth}
		\centering
		\includegraphics[height=0.98\textwidth]{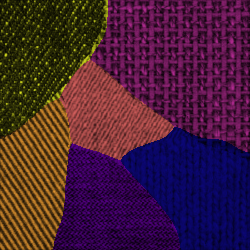}
		\caption[]{Labeling result}\label{fig:intro:res}
	\end{subfigure}
	\caption[]{\label{fig:intro} Labeling of textures based on means and covariance matrices 
	obtained from vectors containing the gray value and the first and second order differences in horizontal and vertical direction.
	Different labels are visualized by different colors in the right image.}
\end{figure}	

As a \emph{first ingredient} for a labeling procedure 
features of the pixels together with a similarity measure for these features is needed.
In this paper we only deal with supervised labeling where prototypical features for each label are available. 
While for natural color images in the RGB color space often
the $\ell_p$-distance, $p\in \{1,2\}$,
on $\mathbb R^3$ is sufficient, this choice is clearly not appropriate when dealing with manifold-valued images
as, e.g., in DT-MRI, where the image pixels are positive definite $3\times 3$ matrices or in Electron Backscattered Diffraction,
where the pixels lie in a quotient manifold of the special orthogonal group.
Also for the texture image in Fig.~\ref{fig:intro} we have to deal with more appropriate features,
actually means and covariance matrices associated to the image pixels as described in the numerical part in Section~\ref{sec:numerics}, paragraph on ,,symmetric positive definite matrices'', second example.
In other words, often the features can be considered to lie on a manifold and 
the distance on the manifold provides a good similarity measure.
We remark that there is a large amount of literature on the collection of appropriate image features
and a detailed description is beyond the scope of this paper.
However, taking just features of single pixels and their distances to  prototypical label features into account
would lead to a high sensibility of the method to errors.
To get robust labeling procedures and to
respect moreover the geometry of the image 
the local or nonlocal neighborhood of the pixels should be involved as a \emph{second ingredient}
into the labeling procedure.
A first idea would be a two step method which applies a neighborhood filtering procedure
directly on the image features followed by a label assignment 
via comparison of distances to the prototypical label features.
The neighborhood filters may be local or nonlocal~\cite{BCM2010,DDT2009,GO07}
and linear or nonlinear~\cite{TM98}. Indeed, such nonlinear methods may also
comprise nonlinear diffusion filter~\cite{BCM2006,We98} or
variational restoration models inspired by the Mumford-Shah functional~\cite{CCZ2013,MS89}.
However, such methods become rather complicated if the features lie for example in a non flat manifold.
Then already the simple local, linear averaging method calls for the computation of Karcher means
which requires itself an iterative procedure. Note that recently nonlocal patch-based methods
were extensively applied for the denoising of manifold-valued images in~\cite{LPS2016}.
In this paper, we will apply the neighborhood filtering not directly on the image features, 
but on the so-called label assignment vectors. 
The label assignment vectors 
$W_i\coloneqq \left(W_{i,k} \right)_{k=1}^K$, $i=1,\ldots,n$,
contain the probabilities $W_{i,k}$ that image pixel $i$
belongs to label $k$. Clearly, since we are dealing with probabilities, we have that $W_i$ 
belongs to the probability simplex
\begin{equation}\label{prob_simplex}
\Delta_K \coloneqq \Bigl\{x \in \mathbb R^K:\sum_{k=1}^K x_k = 1, x \ge 0\Bigr\}.
\end{equation}
Hence, our neighborhood filtering process is independent of the considered feature space. 
The idea to handle neighborhood relations via the label assignment vectors can also be found
in so-called relaxation labeling approaches~\cite{HZ83,Pel97,RHZ1976}. 
However these methods do not use a multiplicative filtering process.

Besides the above mentioned  methods 
there exists a huge number of variational labeling approaches and no single technique works best for all cases.
The models can roughly be divided into continuous and discrete ones.
Recent continuous multilabel models as those in \cite{BYT2011,CCP2012,LS2011} consider a convex functional which consists
of a data term that contains the distances between the pixel features and prototypical features weighted
by the label assignment vectors and a regularization term which enforces certain neighborhood relations
on the label assignment vectors. A frequent choice are isotropic total variation regularizers~\cite{ROF92}
and their relatives as
wavelet-like terms~\cite{HS2013}.
The global minimizers can be found by primal dual first order methods~\cite{BSS2016,CP11}.
In~\cite{CEN2006} it was shown that the relaxed two label method is tight, 
a promising result which is unfortunately no longer true for multilabel tasks~\cite{LLS13}.
Discrete models, in particular relaxed linear programs, are often the method of choice 
in computer vision since they appear to be
more efficient for huge sets of high dimensional data.
Moreover they are in general tighter than continuous models, see exemplary~\cite{Kol06,KZ04}. 
For a comprehensive study of discrete energy minimizing methods also in the context of image segmentation 
we refer to~\cite{Kappes2015} which significantly expands the scope of methods  involved in the earlier
comparative study~\cite{Szeliski2008}.
The above mentioned continuous and discrete models are non-smooth  and convex.
We want to mention that there exist also non convex approaches as, e.g., 
those in~\cite{HH1993,HB1997,Orland1985,WJ2008,YFW2005}
which always have to cope with local minima.

Recently, {\AA}str{\"o}m, Petra, Schmitzer, and Schn{\"o}rr~\cite{APSS2016} suggested an
interesting  supervised geometric approach to the labeling problem, see also \cite{Astroem2016d,Astroem2016a} for the corresponding conference papers. The objective function
to maximize is defined on the open manifold of stochastic matrices with the Fisher-Rao metric and a maximizing algorithm
via the corresponding Riemannian gradient ascent flow is considered.
In the numerical part the authors apply several simplifications, 
in particular lifting maps to cope with the non completeness of the manifold.
Finally this leads to an iterative numerical procedure.
Unlike the continuous Riemannian gradient flow that is shown in~\cite{APSS2016} to converge to edges of the simplex, 
the authors merely showed that the simplified numerical scheme closely approximates this flow, 
but did not prove its convergence. In this paper we discuss the convergence of the resulting numerical scheme.

Since the approach in~\cite{APSS2016} may be hard to read for people not trained in differential geometry 
we show in Section~\ref{sec:mf_mod} how their final iteration scheme can be rewritten 
such that we arrive at an easily understandable algorithm. 
A slight modification of this algorithm leads to our (even simpler) algorithm  reported in the next section.

We start with this algorithm  since we want to introduce it from a user-friendly point of view 
as a multiplicative filtering process of label assignment vectors.
We initialize our iterations with a label assignment matrix containing label assignment vectors of
averaged distances between the prior 
features and the observed ones. Matrices of this kind are common kernel matrices in nonlinear filter~\cite{BCM2010,TM98}, patch- and graph Laplacian based
methods, see, e.g.,~\cite{BN03,EA06}.
In particular the data and their priors may lie in any metric space, which makes the method highly flexible.
Then these label assignment vectors are iterated in a multiplicative way, where we have to take care that the components stay within the probability simplex and
a (small) fixed distance $\varepsilon > 0$ away from zero.
We analyze the convergence of the algorithm if such an $\varepsilon$ is not fixed.
In exact arithmetic it appears that in most relevant cases the algorithm would converge to a trivial
label assignment matrix. This is not the case if we demand that the components are bounded away from zero
by a small distance. Such condition is used in the numerical part anyway to avoid an arithmetic underflow.
However, the analysis shows that there is also a theoretical reason for bounding the components away from zero.
Indeed, in all numerical examples we observed convergence of the algorithm to the vertices of
the $\varepsilon$ probability simplex 
$\Delta_{K,\varepsilon} \coloneqq \{x \in \mathbb R^K:\sum_{k=1}^K x_k = 1, x_i \ge \varepsilon, i=1,\ldots,K\}$.
We reinterpret our algorithm as gradient ascent step of a certain function on the product manifold of
positive numbers followed by a reprojection onto the $\varepsilon$ probability simplex.
The projection is done with respect to the Kullback-Leibler distance.
We provide various numerical examples.

The outline of this paper is as follows: In Section~\ref{sec:mf} we introduce and discuss
our iterative multiplicative filtering algorithm:
the simple algorithm is presented in Subsection~\ref{subsec:basic_alg}, followed by its convergence analysis
in the case that the components  of the label assignment vectors are not bounded away from zero 
in Subsection~\ref{subsec:convergence}.
We interpret the algorithm as a gradient ascent step of a certain function followed by a Kullback-Leibler projection onto
the $\varepsilon$ probability simplex in Subsection~\ref{subsec:gradient}.
The relation to paper~\cite{APSS2016} together is detailed
in Section~\ref{sec:mf_mod}. 
Numerical results demonstrate the very good performance of our algorithm in Section~\ref{sec:numerics}.
In particular we apply the method for the partitioning  of 
manifold-valued images. 
The paper finishes with conclusions in Section~\ref{sec:conclusions}.

%-------------------------------------------------------------------------------
\section{Multiplicative Filtering of Label Assignment Matrices} \label{sec:mf}
%-------------------------------------------------------------------------------
\subsection{Basic Algorithm} \label{subsec:basic_alg}
%-------------------------------------------------------------------------------
In this subsection we introduce the supervised labeling algorithm. 
As outlined in the introduction two basic
ingredients are needed, namely i) feature vectors (prototypical ones for the labels of each pixel)
together with a distance measure between them, and
ii) neighborhood relations between pixels usually given by certain weights.

Let $I_n$ denote the $n\times n$ identity matrix,
$\zb 1_{n,m}$ the $n \times m$ matrix with all components $1$, and
$\zb e_k \in \mathbb R^K$ the $k$-th unit vector in $\mathbb R^K$.
By $\lVert x\rVert_1 \coloneqq |x_1| + \ldots + |x_K|$ we denote the $\ell_1$ norm  and by
$\langle x,y \rangle \coloneqq x_1y_1 + \ldots+x_Ky_K$ the inner product of vectors $x,y \in \mathbb R^K$.
Likewise, if $X,Y \in \mathbb R^{n,K}$ are matrices, we use
$
\langle X,Y\rangle = \sum_{j=1}^n \sum_{k=1}^K X_{j,k} Y_{j,k}
$.
Let $({\mathcal M},\distm)$ be a metric space.
In our applications ${\mathcal M}$ will be $\mathbb R^d$ or a non-flat Riemannian manifold.

We are interested in partitioning $n$ features
\[
	f_i \in {\mathcal M}, \quad i=1,\ldots,n,
\]
into $K \ll n$ labels, where we suppose that we are given $K$ prior features 
\[
	f_k^* \in {\mathcal M}, \quad k=1,\ldots,K.
\]
In other words, we consider supervised labeling.
Then
\begin{equation} \label{dist}
{\mathcal D} \coloneqq (d_{i,k})_{i=1,k=1}^{n,K} 
= 
\begin{pmatrix} {\mathcal D}_1 \, \ldots \, {\mathcal D}_n \end{pmatrix}^\tT, \quad d_{i,k} \coloneqq \distm(f_i,f_k^*),
\end{equation}
is the distance matrix of the features to the priors.
Throughout this paper exponentials and logarithms of matrices are meant
componentwise. Further, $x \circ y$ denotes the componentwise product of
matrices.
To each pixel $i \in \{1,\ldots,n\}$ we assign a neighborhood ${\mathcal N}_\alpha(i) \subseteq \{1,\ldots,n\}$
and set
\begin{equation} \label{eq:initialization}
A_{i} \coloneqq \Exp \Big( - \sum\limits_{j \in {\mathcal N}_\alpha (i)} 
\alpha_{i,j} {\mathcal D}_j \Big)/ \Big\| 
\Exp \Big( - \sum\limits_{j \in {\mathcal N}_\alpha(i)} \alpha_{i,j} {\mathcal D}_j \Big) \Big\|_1 \in \mathbb R_{>0}^K
\end{equation}
and
\[
A \coloneqq (A_1 \, \ldots \, A_n)^ \tT \in \mathbb R_{\ge 0}^{n,K},
\]
where $\alpha_{i,j} \in (0,1)$ with
$\sum_{j \in {\mathcal N}_\alpha(i)} \alpha_{i,j} = 1$.
We can extend the definition of the weights $\alpha_{i,j}$ to all $j=1,\ldots,n$ 
by setting $\alpha_{i,j} = 0$ if $j \not \in {\mathcal N}_\alpha(i)$.

We will work on \emph{label assignment matrices} with non-negative entries 
\[
W \coloneqq \begin{pmatrix} W_1 \, \ldots \, W_n \end{pmatrix}^\tT \in \mathbb R_{\ge 0}^{n,K},
\]
whose rows sum up to 1.
In other words, $W$ consists of \emph{label assignment vectors} $W_i \in \Delta_K$, where $\Delta_K$ denotes the probability simplex~\eqref{prob_simplex}.
In particular, $W_i^{(0)} \coloneqq A_i \in \Delta_K$ will serve as initialization for our algorithm.
Then we apply a multiplicative neighborhood averaging to the columns of $W$, i.e., separately for each label $k \in \{1,\ldots,K\}$.
To have full flexibility, we may consider other neighborhoods ${\mathcal N}_\rho$ and weights 
$\rho_{i,j} > 0$ with 
$\sum_{j \in {\mathcal N}_\rho (i)} \rho_{i,j}  = 1$, $i=1,\ldots,n$
for this process. 
We assume that the weights in Algorithm~\ref{alg:mult} fulfill 
\[
  \rho_{i,j} = \rho_{j,i},
\]
i.e., mutual neighbors are weighted in the same way
and extend their definition to all indices in $\{1,\ldots,n\}$
by setting $\rho_{i,j} \coloneqq 0$ if $j \not \in {\mathcal N}_\rho(i)$.
In summary we propose the Algorithm~\ref{alg:mult}:

%
%-------------------------------------------------------------------------------
\begin{algorithm}[htp]%[tbp]
	\caption{Multiplicative Filtering} 
	\label{alg:mult}
	\begin{algorithmic}
		\State \textbf{Input:}
		$W^{(0)} \coloneqq A$, 
		$\rho_{i,j} \in [0,1]$ with 
		$\sum_{j \in {\mathcal N} (i)} 	 \rho_{i,j}  = 1$, $i=1,\ldots,n$,
		$\varepsilon \coloneqq 10^{-10}$
		\State $r = 0$;
		\Repeat
		\For{$i=1,\ldots,n$}
		\State 
		1. $
		U_i^{(r+1)}
		= W_i^{(r)} \circ \prod\limits_{j \in {\mathcal N}_\rho(i)} (W_j^{(r)} )^{\rho_{i,j}}
		$
		\\
		
		\State
		2. $
		V_i^{(r+1)}
		= U_i^{(r+1)}/\| U_i^{(r+1)} \|_1
		$
		\\
		
		\State
		3. Set $W_i^{(r+1)}\coloneqq V_i^{(r+1)}$ and proceed
		\State 
		${\cal I} = \{k: W_{i,k}^{(r+1)} \le \varepsilon \}$
		\While{
		$\exists k \in \{1,\ldots,K\} : W_{i,k}^{(r+1)} < \varepsilon$}
			\State
			${\cal I} = {\cal I} \cup \argmin_{k=1,\ldots,K} W_{i,k}^{(r+1)}$
			\\
			
			\State
			$
			\tau_{{\cal I}} = \frac{ 1-|{\cal I}| \epsilon} { 1 - \sum_{k \in {\cal I}} W_{i,k}^{(r+1)}}
			$
			\\
			
			\State
			$
			W_{i,k}^{(r+1)} = 
			\left\{
			\begin{array}{ll}
			\varepsilon & k \in {\cal I}
			\\
			 \tau_{{\cal I}} \; 
			 W_{i,k}^{(r+1)} & \text{otherwise}
			 \end{array}
			 \right.
			 $
			\EndWhile	
		
		\EndFor
		\State $r\rightarrow r+1$;		
		\Until a stopping criterion is reached
	\end{algorithmic}
\end{algorithm}
%---------------------------------------------------
%

When the algorithm stops with $W_i^{(R)}$, we assign the label 
\begin{equation} \label{ass_final}
\argmax_{k \in \{1,\ldots,K\}}  W_{i,k}^{(R)}
\end{equation}
to the $i$-th pixel.
Let us comment the steps of the algorithm:
\\
\paragraph{Step 1.}
Here $\prod\limits_{j \in {\mathcal N}_\rho(i)} (W_j) ^{\rho_{i,j}}$ 
is  the  weighted geometric mean of the label assignment vectors
$W_j$ in the neighborhood of $i$. 
In particular we have for $\rho_{i,j} \coloneqq 1/|{\mathcal N}_\rho(i)|$ that
\[
\prod\limits_{j \in {\mathcal N}_\rho(i)} (W_{j,k} )^{\rho_{i,j}} = \Bigl( \prod\limits_{j \in {\mathcal N}_\rho(i)} W_{j,k} \Bigr) ^\frac{1}{|{\mathcal N}_\rho(i)|},
\quad k =1,\ldots,K.
\]
Taking the logarithm we obtain
\[
\log U^{(r)} = \log W^{(r)} + P \log W^{(r)}
\]
with the weight matrix
\begin{equation} \label{P}
 P \coloneqq ( \rho_{i,j} )_{i,j=1}^n.
\end{equation}
Hence Step 1 can be rewritten as
\begin{equation} \label{step1_other}
 U^{(r+1)} =  W^{(r)} \circ \exp\Bigl( P \log W^{(r)} \Bigr).
\end{equation}
\paragraph{Step 2.} Here we ensure that the rows of $V^{(r+1)}$ lie in the probability simplex $\Delta_K$.
\\
\paragraph{Step 3.} This step basically avoids that the entries of $W^{(r)}$ become too small (actually smaller than $\varepsilon = 10^{-10}$).
Very small values do also not appear if we set
\begin{equation} \label{xy}
W_i^{(r+1)} \coloneqq V_i^{(r+1)}+\zb{1}_K\Bigl(\varepsilon-\min_k\{V_{i,k}^{(r+1)}\}\Bigr)
\end{equation}
followed by a normalization.
Indeed, in numerical tests this update instead of Step 3 works  similar.
We have chosen Step 3 since together with Step 2 it has a nice interpretation as
Kullback-Leibler projection onto the part of the probability simplex with entries entries larger or equal to $\varepsilon$.
However, we will see in the next subsection that Step 3 is not just numerical cosmetics, but in general absolutely
necessary to avoid that the iterates $W_i^{(r)}$ converge to the same vertex of the simplex for all $i=1,\ldots,n$
which would result in a trivial labeling.

By Step 1, the weight matrix $P$ plays a crucial role. By definition this matrix is
symmetric and stochastic, i.e., it has nonnegative entries and $P \zb 1_n = \zb 1_n$.
We finish this subsection by collecting some well-known properties of symmetric stochastic matrices, see, e.g.,~\cite{Fro1912,HJ1986,Per1907,W1950}.
%
%-------------------------------------------------------------------------------
\begin{lemma} \label{lem:stochastic}
Let $P \in \mathbb R_{\ge 0}^{n,n}$ be a symmetric stochastic matrix.
Then the following holds true:
\begin{enumerate}[label={\normalfont{\roman*)}}]
 \item
 The matrix $P$ has an eigenvalue decomposition of the form 
  \begin{equation}\label{decomp}
P = Q \Lambda Q^\tT, \qquad \Lambda \coloneqq \Diag (\lambda_1, \ldots, \lambda_n),
\end{equation}
where $\lambda_1 = 1$, $|\lambda_i | \le 1$, $i=2,\ldots,n$ and
\[
Q \coloneqq \begin{pmatrix} q_1 \,  \ldots \, q_n \end{pmatrix}^\tT
\]
is an orthogonal matrix with first column $(q_{j,1})_{j=1}^n = \zb 1_n/\sqrt{n}$.
\item If $P$ has positive diagonal entries, then $-1$ is not an eigenvalue.
\item If $P$ is irreducible, i.e., the associated matrix graph is connected, then $1$ is a single eigenvalue.
\item If $P$ has positive diagonal entries and $\lambda_1 = 1$ is a single eigenvalue, then $P^r$ converges as $r \rightarrow \infty$
to the constant matrix $\frac{1}{n} \zb 1_{n,n}$.
\end{enumerate}
\end{lemma}
%-------------------------------------------------------------------------------
%
%
%-------------------------------------------------------------------------------
\subsection{Convergence Analysis for \texorpdfstring{$\varepsilon = 0$}{ε}} \label{subsec:convergence}
%-------------------------------------------------------------------------------
%
In this subsection, we analyze the convergence of Algorithm~\ref{alg:mult} without Step 3, i.e.,
we use the same initialization as in the algorithm, but
compute just
\begin{align}
1. & \quad U_i^{(r+1)} =
		 W_i^{(r)} \circ \prod\limits_{j \in {\mathcal N}_\rho(i)} (W_j^{(r)} )^{\rho_{i,j}},
\\
2.  &	\quad	W_i^{(r+1)}
		= U_i^{(r+1)}/\| U_i^{(r+1)} \|_1.  \label{alg_triv}
\end{align}
In other words we consider the case $\varepsilon = 0$. 
We will see that the third step is indeed an essential.

We start with the following observation.

%------------------
\begin{lemma}
 The $r$-the iterate in \eqref{alg_triv} is given by
 \begin{equation} \label{iter_0}
 W^{(r)} = \Biggl( \frac{\exp w_1^{(r)}}{\| \exp w_1^{(r)} \|_1}, \ldots,  \frac{\exp w_n^{(r)}}{\| \exp w_n^{(r)} \|_1}\Biggr)^\tT, 
 \end{equation}
with
\[
w^{(r)} = \big( w_1^{(r)}\, \ldots \, w_n^{(r)} \big)^\tT \coloneqq (I_n + P)^r  a, \quad a \coloneqq \log A
\] 
and the weight matrix $P$ in \eqref{P}.
\end{lemma}
%------------------

\begin{proof}
First note that we obtain the same iterates $W_i^{(r)}$ 
if we skip the normalization Step 2 in the intermediate
iterations and normalize only in step $r$.
Therefore we consider the sequence with the same starting point $\tilde W^{(0)} \coloneqq W^{(0)}$ and iterations
\[
		\tilde W_i^{(r+1)}
		\coloneqq \tilde W_i^{(r)} \circ \prod_{j \in {\mathcal N}(i)} \bigl(\tilde W_j^{(r)} \bigr)^{\rho_{i,j}}.
\]
Taking the logarithm and setting
$
w_i^{(r)} \coloneqq \Log \tilde W_i^{(r)}
$
the iteration becomes
\begin{equation}\label{eq:logiter}
w_i^{(r+1)} = w_i^{(r)} + \sum_{j \in {\mathcal N}(i)} \rho_{i,j} w_j^{(r)} ,
\end{equation}
where $w^{(0)} = \log  W^{(0)} = a$.
With $w^{(r)} \coloneqq \big( w_1^{(r)}\, \ldots \, w_n^{(r)} \big)^\tT$ this can be rewritten as
\begin{align}
w^{(r+1)} &= (I_n + P) w^{(r)}\\
&= (I_n + P)^{r+1} w^{(0)}. \label{one}
\end{align}
Then $\tilde W^{(r)} = \exp w^{(r)}$ and row normalization yields the assertion.
\end{proof}

The following remark gives a first intuition about the convergence behavior of $W^{(r)}$.

\begin{remark} \label{rem:conv}
 By \eqref{one} we have 
 \[
 w^{(r)} = 2^r \biggl( \frac12(I_n + P) \biggr)^r w^{(0)} = 2^r v^{(r)}.
 \]
 From the definition of $P$ we know that $\frac12(I_n + P)$ is a symmetric stochastic matrix with positive diagonal entries.
 Assuming that $1$ is a single eigenvalue of this matrix, we conclude by Lemma~\ref{lem:stochastic} 
 that $v^{(r)}$ converges to the matrix
 $\frac{1}{n} \zb 1_n \zb 1_n^\tT w^{(0)} $ with constant columns as $r\rightarrow \infty$.
  However, this does in general not imply that the columns of $W^{(r)}$ also converge to a constant vector for $r\rightarrow \infty$
  as the following example shows:
  We have
  \[
  v^{(r)} \coloneqq 
  \begin{pmatrix}
      \frac12 + \frac{1}{2^r}
      &
      \frac12 + (\frac{2}{3})^r \\[0.5ex]
      \frac12 + (\frac{2}{3})^r
      &
      \frac12 + \frac{1}{2^r}
      \end{pmatrix}
      \quad \rightarrow \quad 
      \begin{pmatrix}
    \frac12&\frac12
    \\[0.5ex]
    \frac12 &\frac12
    \end{pmatrix} \eqqcolon v \quad {\textrm as} \; r\rightarrow \infty.
  \]
  Multiplying with $2^r$, taking the exponential and normalizing the rows leads in the first column of $v^{(r)}$ to
    \begin{align}
      \frac{ \exp 
      \Biggl( 2^r  
      \begin{pmatrix}
       \frac12 + \frac{1}{2^r}
       \\
       \frac12 + (\frac{2}{3})^r
       \end{pmatrix} 
       \Biggr) 
       }
       { \exp \bigl(
         2^r(\frac12 + \frac{1}{2^r}) \bigr)
         +
         \exp\bigl(2^r(\frac12 + (\frac{2}{3})^r
         )\bigr)} 
     &= \frac{\exp \Biggl(  \begin{pmatrix}
            2^{r-1} + 1
            \\
            2^{r-1} + (\frac{4}{3})^r
            \end{pmatrix} \Biggr) }{ \exp\big(
              2^{r-1} + 1\big)
              +
              \exp\big(2^{r-1} + (\frac{4}{3})^r\big)} \\
      & =\frac{\exp
      \Biggl(  \begin{pmatrix}
           1
           \\
       (\frac{4}{3})^r
         \end{pmatrix} \Biggr) }{ \exp
         \big(
       1\big)
         +
       \exp\big((\frac{4}{3})^r\big)}
           \end{align}
           so that
 \[
   \lim_{r \rightarrow \infty} 
   \frac{\exp\bigl(2^r v^{(r)}\bigr)}{\exp\bigl(2^r v^{(r)}\bigr)_{1,1} + \exp\bigl(2^r v^{(r)}\bigr)_{1,2}}
       = \begin{pmatrix}
       0 &1\\ 
       1 &0
       \end{pmatrix}.
 \]
\end{remark}
The general convergence result is given next.
%
%-------------------------------------------------------------------------------
\begin{theorem} \label{thm:conv}
\begin{enumerate}[label={\normalfont{\roman*)}}]
\item
The sequence $\{ W_i^{(r)} \}_r$ generated by~\eqref{alg_triv} converges.
\item
Let $P$ have the eigenvalue decomposition \eqref{decomp}, where $\Lambda$ contains the eigenvalues in descending order.
Assume that there are $M$ pairwise distinct eigenvalues $\lambda_{n_s}$ with multiplicity  $\kappa_{n_s}$, $s=1,\ldots,M$.
Let $\hat s$ be the largest index such that $\lambda_{n_{\hat s}} > 0$. Set 
\[
 \hat a \coloneqq Q^\tT \Log A , \qquad \hat a=(\hat a_1 \, \ldots \, \hat a_K) \in  \mathbb R^{n,K}.
\]
 Then, for every $i \in \{1,\ldots,n\}$, the sequence $\{ W_i^{(r)} \}_r$ 
 converges to a unit vector if and only if $A$ fulfills the following property
 \begin{itemize}
\item[{\normalfont{(PI):}}] If for an $i \in \{1,\ldots,n\}$, there exist $k,l\in \{1,\ldots,K\}$, $l \not = k$
 such that
 \[
 c_{l,k}(s) \coloneqq \sum_{j= n_{s}}^{n_{s} + \kappa_{n_s} -1} q_{i,j} (\hat a_{j, l}- \hat a_{j,k} ) = 0 \qquad \text{for all } s=1,\ldots, \hat s, 
 \]
 then there exists $\tilde l \in \{1,\ldots,K\}$ and  $\tilde s = \tilde s(k,\tilde l) \in \{1,\ldots,\hat s\}$ so that
\begin{equation}
 c_{\tilde l,k}(\tilde s) > 0  \quad \text{and} \quad
 c_{\tilde l,k} (s) = 0 \qquad \text{for all } s = 1,\ldots,\tilde s-1.
 \end{equation}
 \end{itemize}
 \end{enumerate}
\end{theorem}
%-------------------------------------------------------------------------------
%
\begin{proof}
i)
By~\eqref{one} and \eqref{decomp} we obtain
\begin{align}
w^{(r)} &= (I_n + P)^r w^{(0)} = Q (I_n + \Lambda)^r \hat a 
= Q \, D^r \,\hat a, \label{dec_1}
\end{align} 
where $D \coloneqq \Diag( \mu_1, \ldots,\mu_n)$ and $\mu_j = 1 + \lambda_j$, $j=1,\ldots,n$.
By Lemma~\ref{lem:stochastic} and Gershgorin's circle theorem we know that $\mu_j \in (0,2]$, $j=1,\ldots,n$
and $\mu_j \in (1,2]$ if and only if $j\le n_{\hat s}$.
Taking the exponential and normalizing the rows we get
for the $i$-th row
\begin{align}
W_i^{(r)} 
&= 
\frac{1}{ \sum_{k=1}^K  \Exp(q_i^\tT D^r \hat a_k) } 
\left( \Exp(q_i^\tT D^r \hat a_1) \; \ldots \; \Exp(q_i^\tT D^r \hat a_K) \right)^\tT
\\
&=
\left( \frac{1}{
1+ \frac{\Exp(q_i^\tT D^r \hat a_2)}{\Exp(q_i^\tT D^r \hat a_1)}
 + \ldots + 
\frac{\Exp(q_i^\tT D^r \hat a_K)}{\Exp(q_i^\tT D^r \hat a_1)}
} \;  \ldots\;
 \frac{1}{
\frac{\Exp(q_i^\tT D^r \hat a_1)}{\Exp(q_i^\tT D^r \hat a_K)}
 + \ldots + 
\frac{\Exp(q_i^\tT D^r \hat a_{K-1})}{\Exp(q_i^\tT D^r \hat a_K) } + 1
}
\right)^\tT\\
&=
\left( \frac{1}{1+G_{2,1}(r) + \ldots + G_{K,1}(r)} \; \ldots \; \frac{1}{G_{1,K}(r) + \ldots + G_{K-1,K}(r) + 1} \right)^\tT,
\label{struktur}
\end{align}
where 
\begin{equation} \label{wichtig}
G_{l,k} (r) \coloneqq \frac{\Exp(q_i^\tT D^r \hat a_l)}{\Exp(q_i^\tT D^r \hat a_k)} = \prod_{j=1}^n \exp \left(q_{i,j}  \mu_j^r (\hat a_{j,l} - \hat a_{j,k} ) \right).
\end{equation}
Taking the logarithm results in
\begin{align*}
g_{l,k}(r) &\coloneqq \Log\bigl(G_{l,k}(r)\bigr) 
= \sum_{j=1}^n q_{i,j}  (\hat a_{j,l} - \hat a_{j,k} ) \mu_j^r 
=  \sum_{s=1}^{M} c_{l,k}(s) \mu_{n_s}^r\\
&=\sum_{s=1}^{\hat s} c_{l,k}(s) \mu_{n_s}^r + \sum_{s=\hat s+1}^{M} c_{l,k}(s) \mu_{n_s}^r.
\end{align*}
The first sum contains the eigenvalues larger than 1 and the second one those smaller or equal to 1.
If it exists, let $\tilde s \in \{1,\ldots,\hat s\}$ be the smallest index such that 
$c_{l,k}(\tilde s) \not = 0$. Then it follows
\[
\lim_{r \rightarrow \infty} g_{l,k} (r) =
\left\{
\begin{array}{ll}
 \mathrm{const} &\mathrm{if } \; c_{l,k}(s) = 0 \; \mbox{for all } \; s = 1,\ldots,\hat s, \\
  \sgn(c_{l,k}(\tilde s)) \, \infty &\mathrm{if } \; c_{l,k}(s) = 0 \; \mbox{for all } \; s = 1,\ldots,\tilde s - 1, 
\end{array}
\right.
\]
and consequently
\begin{equation} \label{wow}
\lim_{r \rightarrow \infty} G_{l,k} (r) =
\left\{
\begin{array}{ll}
 \exp(\mathrm{const}) &\mathrm{if } \; c_{l,k}(s) = 0 \; \mbox{for all } \; s = 1,\ldots,\hat s, \\
  +\infty &\mathrm{if } \; c_{l,k}(\tilde s) > 0 , \\
  0&\mathrm{if } \; c_{l,k}(\tilde s) < 0.
\end{array}
\right.
\end{equation}
Hence, by \eqref{struktur}, we see that $\{ W_i^{(r)}\}_r$ converges as $r\rightarrow \infty$.

ii) It remains to examine under which conditions it converges to a unit vector.
Consider a fixed $k \in \{1,\ldots,K\}$. We distinguish two cases.

1. If for all $l \in \{1,\ldots,K\} \backslash \{k\}$ there exists an $s = s(l) \in \{1,\ldots,\hat s\}$
such that $c_{l,k}(s) \not = 0$, then we see by \eqref{wow} that $\lim_{r \rightarrow \infty} W_{i,k}^{(r)} \in \{0,1\}$.
Since, by construction the components of $W_i^{(r)}$ sum up to 1, we conclude that $W_i^{(r)}$ converges to a unit vector.

2. If there exists $l \in \{1,\ldots,K\} \backslash \{k\}$ such that $c_{k,l} (s)  =0$ for all $s=1,\ldots, \hat s$,
then, by \eqref{wow}, the component $W_{i,k}^{(r)} = (1+ G_{l,k}(r) + \sum_{m \not = l,k} G_{m,k}(r) )^{-1}$ is strictly smaller than 1 and becomes
0 if and only if there exists $\tilde l \in \{1,\ldots,K\}$ and  $\tilde s = \tilde s(k,\tilde l) \in \{1,\ldots,\hat s\}$ so that
$ c_{\tilde l,k}(\tilde s) > 0$ and $c_{\tilde l,k} (s) = 0$ for all $s = 1,\ldots,\tilde s-1$.
This is exactly condition (PI) and we are done.
\end{proof}

Let us have a closer look at the property (PI).

If ${\mathcal N}(i) = \{i\}$ for all $i \in \{1,\ldots,n\}$ which means that we do not take the neighborhoods of the points into account, 
then $P = I$ and property {\normalfont{(PI)}} is fulfilled if for
every $i \in \{1,\ldots,n\}$ the vector $\bigl(\mathrm{dist}(f_i,f_k^*)\bigr)_{k=1}^K$ has a unique smallest element.
Then the algorithm gives the same assignment as
just ${\argmin}_k \operatorname{dist}(f_i,f_k^*)$ which is reasonable.

Things change completely if we take care of the point neighborhoods.
First we give an equivalent condition to (PI).

\begin{remark} \label{rem:PI} 
Having a look at part ii) of the above proof we realize by \eqref{wow} and \eqref{struktur} 
that $\{ W_i^{(r)}\}_r$ converges to a unit vector
if and only if one of its components converges to 1.
By \eqref{wow} and \eqref{struktur} this is the case if and only if the following condition \textup{(PI$'$)} is fulfilled:
\begin{itemize}
\item[{\normalfont{(PI$'$):}}] 
For all $i=1,\ldots,n$ there exists $k \in \{1,\ldots,K\}$ such that for all $l \in \{1,\ldots,K\} \backslash \{k\}$ there exists $\tilde s = \tilde s(l) \in \{1,\ldots,\hat s\}$
such that
\[
 c_{l,k}(\tilde s) < 0 \quad \mathrm{and} \quad c_{l,k}(s) = 0 \quad \mbox{for all} \quad s=1,\ldots \tilde s - 1.
 \]
 \end{itemize}
Checking condition  \textup{(PI$'$)} would give us the limit vector of $\{ W_i^{(r)}\}_r$ since it provides the component $k$, where the vector is 1.
\end{remark} 

The following corollary shows that in many practical relevant image labeling tasks 
the iterates in \eqref{alg_triv} converge the \emph{same} vertex of the simplex.

\begin{corollary}\label{leider}
 Let $P$ be a symmetric stochastic matrix which has eigenvalue 1 of multiplicity 1 and all other eigenvalues
 have absolute values smaller than 1.
 Assume that there exists $k \in \{1,\ldots,K\}$  such that
 \begin{equation} \label{cond_leider}
 \prod_{j=1}^n A_{j,k} > \prod_{j=1}^n A_{j,l} \quad  \text{for all} \quad l \in \{1,\ldots,K\}\backslash\{k\}.
 \end{equation}
 Then, \emph{for all} $i=1,\ldots,n$, the vectors $W_i^{(r)}$ generated by \eqref{alg_triv} converge for $r \rightarrow \infty$ to the $k$-th unit vector. 
\end{corollary}

\begin{proof}
By assumption on $P$ we have in (PI), resp. (PI$'$), for $s = 1$ and all $k,l \in \{1,\ldots,K\}$ that
\[
c_{l,k}(1) = q_{i,1} (\hat a_{1,l} - \hat a_{1,k})
\]
where
$\hat a = Q^\tT \log A$.
By Lemma \ref{lem:stochastic} i) we know that the first column of $Q$ is given by $(q_{i,1})_{i=1}^n = \zb 1_n/{\sqrt{n}}$.
Thus, 
\[
  \hat a_{1,m}  = \frac{1}{\sqrt{n}} \sum_{j=1}^n \log A_{j,m}.
\]
Choosing $k$ according to 
assumption \eqref{cond_leider} we conclude 
$\hat a_{1,l} - \hat a_{1,k} < 0$ for all $l \in \{1,\ldots,K\} \backslash \{k\}$.
Thus,
\[
c_{l,k}(1) = q_{i,1} (\hat a_{1,l} - \hat a_{1,k})  = (\hat a_{1,l} - \hat a_{1,k})/\sqrt{n}  < 0 \quad for \; all \quad l \in \{1,\ldots,K\}\backslash\{k\}
\]
and we can take $\tilde s =1$ in (PI$'$). Then we see by the proof of Theorem~\ref{thm:conv}, resp. Remark~\ref{rem:PI},
that $W_i^{(r)}$ converges for $r \rightarrow \infty$ to the $k$-th unit vector independent of $i$.
\end{proof}

By Lemma \ref{lem:stochastic} i)-iii), the assumptions on $P$ are fulfilled if
$P$ is irreducible and has positive diagonal entries.
In our task this is the case if we use local pixel neighborhoods,
where the central pixels have positive weights $\rho_{i,i} > 0$, $i=1,\ldots,n$.
Assumption \eqref{cond_leider} is fulfilled in many practical relevant cases,
in particular in all our numerical examples.
This underlines the relevance of Step 3 in Algorithm~\ref{alg:mult}.
We finish this subsection by an intuitive example to illustrate the importance of $\varepsilon$.

\begin{example} \label{ex:leider}
\begin{table}
    \centering
    \begin{tabular}{l|cccccccccc}
        $\varepsilon$ & \multicolumn{10}{c}{Labels}\\
        \hline \\[-2.5ex]
        $10^{-10}$ & 1 & 1 & 3 & 3 & 3& 3 & 3 & 2 & 2& 2\\
        $10^{-11}$ & 1 & 1 & 1 & 3 & 3& 3 & 3 & 2 & 2& 2\\
        $10^{-81}$ & 1 & 1 & 1 & 3 & 3& 3 & 3 & 3 & 2& 2\\
        $0$ & 3 & 3 & 3 & 3 & 3& 3 & 3 & 3 & 3& 3\\
    \end{tabular}
    \caption[]{Labeling of a signal with $A = W^{(0)}$ from \eqref{eq:iniA} using different values of $\varepsilon$.}\label{fig:epsilon}
\end{table}

We initialize Algorithm~\ref{alg:mult} with the matrix
\begin{equation}
A \coloneqq
\left(
\begin{array}{cccccccccc}
\frac{4}{5} & \frac{4}{5} & \frac{1}{5}& \frac{1}{5}& \frac{1}{5}& \frac{1}{5}& \frac{1}{5}&\frac{1}{4}&\frac{1}{4}&\frac{1}{4}\\[0.5ex]
\frac{1}{10} & \frac{1}{10} & \frac{1}{5}& \frac{1}{5}&\frac{1}{5}& \frac{1}{5}& \frac{1}{5}&\frac{1}{2}&\frac{1}{2}&\frac{1}{2}\\[0.5ex]
\frac{1}{10} & \frac{1}{10} & \frac{3}{5}& \frac{3}{5}& \frac{3}{5}&\frac{3}{5}& \frac{3}{5}&\frac{1}{4}&\frac{1}{4}&\frac{1}{4}
\end{array}
\right)^\tT\!\!\!, \label{eq:iniA}
\end{equation}
which corresponds to a labeling of a signal of length $10$ with $3$ features and three labels. 
We choose neighborhoods of size 3 with uniform weights and mirror boundary conditions.
Table~\ref{fig:epsilon} depicts the results of the labeling for the different $\varepsilon$,
where the label assignment changes. In all cases the corresponding label assignment vectors
are the vertices of the $\varepsilon$ probability simplex (up to machine precision)
and the labels 1,2,3 were assigned via \eqref{ass_final}.
Note that we get the same labeling for $\varepsilon = 10^{-1}$ as for $\varepsilon = 10^{-10}$,
but the label assignment matrix converges in the first case to  vertices as $(0.8,0.1,0.1)^\tT$. 
Up to $\varepsilon = 10^{-10}$
the pixels are labeled as expected. 
If $\varepsilon$ decreases further, we still get  reasonable results originating from the large starting values of the first segment.
Until double precision is reached, i.e., $\varepsilon = 10^{-323}$,
we see no more changes in the labeling result. 
To be able to handle $\varepsilon = 0$ we look at $w = \log(W)$ without normalization, 
i.e., we iterate \eqref{eq:logiter}. Then we observe after $95$ iterations that the maximal value of $w^{(95)}$ is attained in the last row. 
This means that after normalization in exact arithmetic all pixels will get the label 3.
Since \eqref{cond_leider} reads for our example as
\begin{equation}
\Bigl(\prod_{j=1}^n A_{j,k}\Bigr)_{k=1}^3 = \frac{1}{2 \times 10^7}
\begin{pmatrix}
    64, 4, 243
\end{pmatrix},
\end{equation}
and the third value is the largest one this agrees with Corollary \ref{leider}.
\end{example} 
%-------------------------------------------------------------------------------
\subsection{Gradient Ascent Reprojection Algorithm} \label{subsec:gradient}
%-------------------------------------------------------------------------------
%
In this subsection we show that Algorithm~\ref{alg:mult} can be seen 
as gradient ascent reprojection algorithm of a certain function
defined on the product manifold of positive numbers.
The positive numbers $\mathbb R^* \coloneqq \mathbb R_{ >0}$ form together with the metric 
on the tangent space $T_x \mathbb R^* = \mathbb R$ given by 
$\langle \xi_1,\xi_2 \rangle_x =  \xi_1 \xi_2/x^2$
a Riemannian manifold. The distance between $x_1,x_2 \in \mathbb R^*$ is defined by
\[
  \operatorname{dist}_{\mathbb R^*} (x_1,x_2) \coloneqq |\log x_1 - \log x_2|.
\]
Indeed, computation in $\mathbb R^*$ can often be reduced to the Euclidean setting
after taking the logarithm. 
The exponential map $\MExp_x\colon T_x \mathbb R^* \rightarrow \mathbb R^*$ 
and its inverse $\MLog_x \colon R^* \rightarrow T_x \mathbb R^*$
read
\begin{equation} \label{EXP}
\MExp_x( \xi) = x \exp \frac{\xi}{x} , \quad \MLog_x y = x \log \frac{y}{x}.
\end{equation}
For $(\gamma_j)_{j=1}^n \in \Delta_n$ and $x_j \in  \mathbb R^*$, $j=1,\ldots,n$ 
the \emph{weighted Riemannian center of mass} on $\mathbb R^*$ is given by
\begin{equation} \label{karcher}
\argmin_{x \in \mathbb R^*} \sum_{j=1}^n \gamma_j \operatorname{dist}_{\mathbb R^*} (x,x_j)^2.
\end{equation}
The minimizer can be computed as follows.
The Riemannian gradient of the squared distance is 
\begin{equation} \label{rg}
\nabla_{\mathbb R^*}\operatorname{dist}_{\mathbb R^*}^2 (\cdot,y) (x) = - 2 \MLog_x y = -2 x \log \frac{y}{x}
= 2 \KL(x,y),
\end{equation}
where $\KL$ denotes the \emph{Kullback-Leibler distance}.
It is defined by 
\begin{align}
\KL\colon \mathbb R^N_{\ge 0} \times \mathbb R^N_{> 0} \rightarrow \mathbb R, \quad
\KL(x,y) \coloneqq \langle x,\log \frac{x}{y} \rangle ,
\end{align}
with the convention $0 \log 0 = 1$ and using the component-wise division in $\tfrac{x}{y}$.
The function is jointly convex in both arguments.
It is the Bregman distance of the Shannon entropy and is related to entropy regularized Wasserstein distances, which have
recently found applications in image processing, see~\cite{Pey2015} and the references therein.
Setting the Riemannian gradient of the objective in \eqref{karcher} to zero using \eqref{rg}, and dividing by 2 we get 
\begin{align} \label{karcher_1}
   0 =  \sum_{j=1}^n\gamma_j x \log \frac{x}{x_j} =  \sum_{j=1}^n \gamma_j \KL(x,x_j) =  \KL\Bigl(x,\prod_{j=1}^n x_j^{\gamma_j} \Bigr).
\end{align}
Consequently, the weighted Riemannian center of mass is just the weighted geometric mean $\operatorname{GM} \coloneqq \prod_{j=1}^n x_j^{\gamma_j}$.

By ${\cal M} \coloneqq (\mathbb R^*)^{nK}$ we denote the usual product manifold. 
We are interested in the function $F\colon{\cal M} \rightarrow \mathbb R$ given by 
\begin{align} \label{functional}
F(W) &\coloneqq\langle \log W , P \log W \rangle \\
&= \sum_{k=1}^K \langle \log \ubar W_k,P\log \ubar W_k \rangle = \sum_{k=1}^K (\log \ubar W_k)^\tT P\log \ubar W_k,
\end{align}
where $\ubar W_k$, $k=1,\ldots,K,$ are the columns of $W$ and $P$ is the weight matrix in \eqref{P}.

In Step 1 of our algorithm we compute $w = \log W$, perform a gradient ascent step of $f(w) = \langle w, Pw\rangle$ in the Euclidean arithmetic
and take the exponential of the result. The direct execution on ${\cal M}$ is given by the following lemma.

%-------------------------------------------------------------------------------
\begin{lemma} 
 Step 1 in Algorithm~\ref{alg:mult} is a gradient ascent step of $F$ given by \eqref{functional}
 at $W^{(r)}$  with step length 1.
 In other words, we compute
 \[
   \MExp_W \big(\nabla_{{\cal M}} F(W) \big)	
     = W \circ \exp\left( P \log W \right)
 \]
 at $W = W^{(r)}$, where $\nabla_{{\cal M}}F$ denotes the Riemannian gradient of $F$.
\end{lemma}
%-------------------------------------------------------------------------------
\begin{proof}
By separability of $F$ we can restrict our attention to single columns $\ubar{W} = \ubar{W}_k \in (\mathbb R^*)^n$ of $W$.
A tangent vector $\xi \in T_{\ubar{W}} (\mathbb R^*)^n$ acts on $F$ as
\[\xi F(\ubar{W}) = \langle \nabla F(\ubar{W}) ,\xi \rangle,\]
where 
$\nabla F(\ubar{W}) = \frac{1}{\ubar{W}} \circ P \log \ubar{W}$ is just the Euclidean gradient of $F$ at $\ubar{W}$.
The Riemannian gradient fulfills
\[
\langle \nabla_{{\cal M}} F(\ubar{W}), \xi \rangle_{\ubar{W}} = \langle \nabla F(\ubar{W}) ,\xi \rangle = \langle \frac{1}{\ubar{W}} \circ P \log \ubar{W} ,\xi \rangle
\]
for all $\xi \in  \mathbb R^n$.
By the above definition of the Riemannian metric on $\mathbb R^*$ we get
\[
\langle \nabla_{{\cal M}} F(\ubar{W}), \xi \rangle_{\ubar{W}} = 
\big\langle \frac{1}{\ubar{W}^2} \circ \nabla_{{\cal M}} F(\ubar{W}), \xi \big\rangle  = \big\langle \frac{1}{\ubar{W}} \circ P \log \ubar{W} ,\xi \big\rangle
\]
so that $\nabla_{{\cal M}} F(\ubar{W}) = \ubar{W} \circ P \log \ubar{W}$. Using finally \eqref{EXP} gives
$\MExp_{\ubar{W}} \big(\nabla_{{\cal M}} F(\ubar{W}) \big) = \ubar{W} \circ \exp \left(P \log \ubar{W} \right)$.
\end{proof}

Next we will see that Step 2 and 3 of our algorithm are just projections of the columns of $U^{(r)}$ onto the 
$\varepsilon$ \emph{probability simplex} defined by
\begin{equation} \label{eps-simplex}
 \Delta_{K,\varepsilon} \coloneqq \Bigl\{x \in \mathbb R^K:\sum_{k=1}^K x_k = 1, x \ge \varepsilon\Bigr\} \subset (\mathbb R^*)^K, \quad \varepsilon > 0.
\end{equation}
Here we do not use orthogonal projections, but projections with respect to
Kullback-Leibler distance.
For the interpretation of the Steps 2-3 in our algorithm as  $\KL$ projections we need the following lemma.
%-------------------------------------------------------------------------------
\begin{lemma} \label{lem:prop_kl}
 \begin{enumerate}[label={\normalfont{\roman*)}}]
    \item
    Assume that $0 < y_1 \le \ldots \le y_K$.  Then, for $0 < \varepsilon < \frac{1}{K}$, the $\operatorname{KL}$ projection of $y$ onto $\Delta_{K,\varepsilon}$
  is given by
\[
\argmin_{x \in \Delta_{K,\varepsilon}} \KL(x,y) = 
(\underbrace{\varepsilon, \ldots, \varepsilon}_{m}, \tau_m y_{m+1} , \ldots, \tau_m y_{K} )^\tT,
\]
where $\tau_m \coloneqq \frac{1 - m \varepsilon}{\|y\|_1 - \sum_{k=1}^m y_k}$ and $m$ is an index such that
\begin{equation} \label{conds}
 y_m\tau_m \le \varepsilon \quad \text{and} \quad y_{m+1} \tau_m > \varepsilon.
\end{equation}
 \item
  For $\varepsilon = 0$ we have
  \[
  \argmin_{x \in \Delta_K} \KL(x,y) =  y /\|y\|_1.\]
\end{enumerate}
\end{lemma}
%-------------------------------------------------------------------------------
\begin{proof}
i) 
We reformulate the convex problem in ii) as
\[
\argmin_{x \in \mathbb R^K} \KL(x,y) \quad \mbox{subject to} \quad \langle x, \zb 1_K \rangle = 1, \; x - \varepsilon \zb 1_K \ge 0.
\]
Using the Lagrangian
\[
L(x,p,\tilde p) \coloneqq \KL(x,y) + \tilde p (\langle \zb 1_K,x \rangle -1) + \langle p,x - \varepsilon \zb 1_K\rangle
\]
the (necessary and sufficient) optimality conditions for a minimizer read 
\begin{align}
& (1) \; \; \nabla_x L(x,p,q) = \log \frac{x}{y} + \zb 1_K + \tilde p \zb 1_K + p = 0,\\
& (2) \; \; \langle x, \zb 1_K \rangle = 1, \qquad (3) \; \; x - \varepsilon \zb 1_K \ge 0,\\
& (4) \; \;  p \le 0, \qquad (5)\;  \; \langle x - \varepsilon \zb 1_K,p \rangle = 0. 
\end{align}
The first condition can be rewritten as
\begin{equation} \label{firstKKTa}
x = \exp(-\tilde p - 1) \, y \circ \exp(-p)
\end{equation}
and together with the second condition and $q \coloneqq \exp(- p)$ we obtain
\begin{equation} \label{firstKKT}
x = \frac{1}{\langle y \circ q,\zb 1_K\rangle} \, y \circ q 
= 
\frac{1}{\langle \frac{y}{\|y\|_1} \circ q,\zb 1_K\rangle} \, \frac{y}{\|y\|_1} \circ q.
\end{equation}
By the last equality we can continue our considerations with a vector $y$ fulfilling $\|y\|_1 = 1$. 

Assume that $x_1 = \ldots = x_m = \varepsilon$ and  $x_{k} >  \varepsilon$ for $k=m+1,\ldots,K$.
Then we have by the fifth condition that $q_k = 1$ for $k=m+1,\ldots,K$ so that
\[
x = \frac{1}{\sum_{k=1}^m q_k y_k + s_m} \left(q_1 y_1,\ldots , q_m y_m, y_{m+1},\ldots,y_K\right)^\tT,
\]
where $s_m = \sum_{k=m+1}^K y_k = 1 - \sum_{k=1}^m y_k$.
This implies $q_ky_k = c$ for all $k=1,\ldots,m$.
Now $x_k = \frac{c}{mc+s_m} = \varepsilon$ gives $c = \frac{s_m \varepsilon}{1-m \varepsilon} = \frac{\varepsilon}{\tau_m}$
so that
\[q_k = \frac{\varepsilon}{\tau_m y_k}, \quad k=1,\ldots,m.\]
Since the $y_k$ are non-descending, we see that the fourth optimality condition 
$q_1 \ge \ldots \ge q_m = \frac{\varepsilon}{\tau_m y_m}  \ge 1$
is fulfilled if the first inequality in \eqref{conds} holds true.
Finally, we have 
\[
\varepsilon < x_{m+1} = y_{m+1} \tau_m \le \ldots \le y_{K} \tau_m = x_K
\]
if the second inequality in  \eqref{conds} is satisfied.
Thus the vector determined by ii) fulfills all optimality conditions and we are done.
\\
ii) This assertion  follows directly from i).
\end{proof}
As an immediate consequence of the previous lemma and \eqref{karcher_1} we obtain the following corollary.

\begin{corollary} \label{alg_step2_3}
 The Steps 2 and 3 in Algorithm~\ref{alg:mult} are the KL projection of $U_i^{(r)}$, $i=1,\ldots,n$
 onto $\Delta_{K,\varepsilon}$. The whole iteration in Step 1-3 can be rewritten as 
\[
 W_i^{(r+1)} = \argmin_{W \in \Delta_{K,\varepsilon}} \sum_{j \in {\mathcal N}_\rho (i)} \rho_{i,j} \KL(W, W_i^{(r)} \circ W_j^{(r)}).
\]
\end{corollary}

%-------------------------------------------------------------------------------
\section{Relation to the Numerical Scheme in~\texorpdfstring{\cite{APSS2016}}{[4]}} \label{sec:mf_mod}
%-------------------------------------------------------------------------------
%
In this section we want to demonstrate how Algorithm~\ref{alg:mult}
is related to an iterative scheme proposed by {\AA}str{\"o}m, Petra, Schmitzer and Schn\"orr~\cite{Astroem2016d,APSS2016,Astroem2016a}
without stretching  details on the geometry of the probability simplex  which can be found in~\cite{APSS2016}.
The authors consider the interior $\mathring \Delta_K$ of the probability simplex~\eqref{prob_simplex},
which together with the Fisher-Rao metric forms a non complete Riemannian manifold. 
The tangent space $T_x{\mathring \Delta_K}$ of $\mathring \Delta_K$
at $x \in \mathring \Delta_K$ is given by the vectors
in $\mathbb R^K$ whose components sum up to zero.
Let ${\mathcal S} \coloneqq (\mathring \Delta_K)^n$ be the corresponding product manifold.
By ${\MExp_x}\colon T_x{\mathcal S} \rightarrow {\mathcal S}$
we denote the manifold exponential map which, by the incompleteness of the manifold, is only defined on a subset of  $T_x{\mathcal S}$.
Then they consider the function $E\colon  {\mathcal S} \rightarrow \mathbb R$ given by
\[
	E(W) \coloneqq  \langle S(W),W \rangle,
\]
where
$S_i(W)$ is the Riemannian center of mass (Karcher mean) of $\{ L_j \in \mathring \Delta_K: j \in {\mathcal N} (i) \}$
and
\[
L(W) \coloneqq \MExp_W (-U), \quad U_i \coloneqq {\mathcal D}_i^\tT - \frac{1}{K} \langle \zb 1_K, {\mathcal D}_i \rangle \zb 1_K
\]
with the distance matrix ${\mathcal D}$ given by \eqref{dist}. 
Since the problem does not admit a closed form solution, 
the authors propose to apply 
a Riemannian gradient ascent flow on ${\mathcal W}$.
In their final numerical scheme they utilize certain simplifications, 
in particular a lifting map for the manifold exponential mapping to cope with the incompleteness of $\mathcal{S}$.
Finally they arrive at the following iterative scheme:
\begin{align}
W_i^{(0)} &\coloneqq  \frac{1}{K} \zb 1_K, \label{it_ch_0}\\
L(W_j^{(r)}) &\coloneqq  \frac{\exp(-{\mathcal D}_j)\circ W_j^{(r)}}{\langle \exp(-{\mathcal D}_j),W_j^{(r)} \rangle},\label{it_ch_1}\\
S_i(W^{(r)}) &\coloneqq 
\frac{ \prod_{j \in {\mathcal N}_{\rho} (i) } L(W_j^{(r)})^{\rho_{i,j}} }{ \langle \zb 1_K, \prod_{j \in {\mathcal N}_{\rho}(i) } L(W_j^{(r)})^{\rho_{i,j}} \rangle},\label{it_ch_2}\\
W_i^{(r+1)} &\coloneqq \frac{ W_i^{(r)} \circ S_i(W^{(r)}) }{ \langle W_i^{(r)},S_i(W^{(r)}) \rangle} \label{it_ch_3}.
\end{align}
where they use the weights $\rho_{i,j}\coloneqq 1/N_i$, $N_i \coloneqq |{\mathcal N}_{\rho}(i)|$.

How does this relate to our setting?
Plugging $L(W_j^{(r)})$ into the expression for $S_i(W^{(r)})$ we obtain
\begin{align*}
S_i(W^{(r)}) 
&=  \frac{ 
\prod_{j \in {\mathcal N}_{\rho}(i) } 
\left( 
\frac{\exp(-{\mathcal D}_j)\circ W_j^{(r)}}{\langle \exp(-{\mathcal D}_j),W_j^{(r)} \rangle} 
\right)^{\rho_{i,j}}
}
{
\left\langle 
\zb 1_K,  \prod_{j \in {\mathcal N}_{\rho}(i)}
\left( \frac{\exp(-{\mathcal D}_j)\circ W_j^{(r)}}{\langle \exp(-{\mathcal D}_j),W_j^{(r)} \rangle}  \right)^{\rho_{i,j}} 
\right\rangle
}\\
&=
 \frac{ 
  \prod_{j \in {\mathcal N}_{\rho}(i) } \left( \exp(-{\mathcal D}_j)\circ W_j^{(r)} \right)^{\rho_{i,j}}}
 { \Big\lVert
  \prod_{j \in {\mathcal N}_{\rho}(i) } \left( \exp(-{\mathcal D}_j)\circ W_j^{(r)} \right)^{\rho_{i,j}} \Big\rVert_1}\\
  &=
  \frac{ \prod_{j \in {\mathcal N}_{\rho}(i) } \exp(-{\mathcal D}_j)^{\rho_{i,j}} \circ \prod_{j \in {\mathcal N}_{\rho}(i) } \left( W_j^{(r)} \right)^{\rho_{i,j}}}
  {\big\lVert \prod_{j \in {\mathcal N}_{\rho}(i) } \exp(-{\mathcal D}_j)^{\rho_{i,j}}    \circ \prod_{j \in {\mathcal N}_{\rho}(i) } \left(W_j^{(r)} \right)^{\rho_{i,j}}\big\rVert_1}.
\end{align*}
Using
$A_i \coloneqq  \prod_{j \in {\mathcal N}_{\rho}(i) } \exp(-{\mathcal D}_j)^{\rho_{i,j}}/\|\prod_{j \in {\mathcal N}_{\rho}(i) } \exp(-{\mathcal D}_j)^{\rho_{i,j}}\|_1$,
i.e., $\rho_{i,j} = \alpha_{i,j}$ in \eqref{eq:initialization}, we get
\begin{align*}
 S_i(W^{(r)}) 
    &= c \,  A_i \circ \prod_{j \in {\mathcal N}_{\rho}(i) } \left( W_j^{(r)} \right)^{\rho_{i,j}},
  \quad c\coloneqq \frac{\big\lVert\prod_{j \in {\mathcal N}_{\rho}(i) } 
  \exp(-{\mathcal D}_j)^{\rho_{i,j}}\big\rVert_1}{\big\lVert \prod_{j \in {\mathcal N}_{\rho}(i) } \exp(-{\mathcal D}_j)^{\rho_{i,j}}    
  \circ \prod_{j \in {\mathcal N}_{\rho}(i) } \left(W_j^{(r)} \right)^{\rho_{i,j}}\big\rVert_1}
\end{align*}
Substituting this into  \eqref{it_ch_3} and dividing the denominator and numerator by the constant $c$ 
we arrive finally at
\begin{equation} \label{V}
V_i^{(r+1)}
= \frac{ A_i \circ W_i^{(r)} \circ \prod\limits_{j \in {\mathcal N}_{\rho}(i)} (W_j^{(r)} )^{\rho_{i,j}}}{\| A_i \circ W_i^{(r)} 
\circ \prod\limits_{j \in {\mathcal N}_{\rho}(i)} (W_j^{(r)} )^{\rho_{i,j}}\|_1}.
\end{equation}
Up to the additional factor $A_i$ formula \eqref{V} coincides with the iterations in Algorithm~\ref{alg:mult}.		

What about the starting point?
Starting with \eqref{it_ch_0} we obtain after the first step  in \eqref{it_ch_3} the matrix
\begin{align*}
W_i^{(1)} &= 
\frac{A_i \circ (\zb 1_K/N_i) \circ  \prod_{j \in {\mathcal N}_{\rho}(i) } \left( \zb 1_K/N_i \right)^{\rho_{i,j}} }
{\Big\| A_i \circ (\zb 1_K/N_i) \circ  \prod_{j \in {\mathcal N}_{\rho}(i) }\left( \zb 1_K/N_i \right)^{\rho_{i,j}}\Big\|_1}
 = \frac{A_i}{\|A_i\|_1} = A_i.
\end{align*}
Therefore we can also start with $A_i$ as in Algorithm~\ref{alg:mult}. 
Clearly, we can use again more flexible weights. 
In summary the numerical scheme in~\cite{APSS2016} has the first step 
\begin{equation} \label{step_1_ch}
		U_i^{(r+1)}
		= A_i \circ W_i^{(r)} \circ \prod\limits_{j \in {\mathcal N}(i)} (W_j^{(r)} )^{\rho_{i,j}}.
\end{equation}
%--------------------------------------------------
The convergence analysis for $\varepsilon = 0$ is basically the same as in Subsection~\ref{subsec:convergence}.
For sake of completeness we add it below.
Again we emphasize that that bounding the values away from zero as, e.g., in the third step of our algorithm
is not just numerical cosmetics. Without such a step the iterates converge similar as in Corollary \ref{leider} to the same vertex of the 
probability simplex which results in a trivial labeling.

%------------------
\begin{corollary} \label{cor:conv}
The sequence of iterates $\{ W_i^{(r)} \}_r$, $i=1,\ldots,n$ generated by 
\eqref{step_1_ch} converges as $r\rightarrow \infty$.
It converges to a unit vector if and only if
$A$ fulfills {\normalfont{(PI)}} but with 
 \[
 c_{l,k}(s) \coloneqq \sum_{j= n_{s}}^{n_{s} + \kappa_s -1} q_{i,j} \lambda_j^{-1} (\hat a_{j, l}- \hat a_{j,k} ).
 \]
\end{corollary}
%-------------------------------------------------------------------------------
\begin{proof}
 The proof follows the same lines as those of Theorem~\ref{thm:conv} so that we highlight only the differences.
 Instead of \eqref{dec_1} we obtain
 \begin{align*}
 w^{(r)} &= \left( I +(I+P) + \ldots + (I + P)^r \right) w^{(0)} = P^{-1} \left( (I+P)^{r+1} - I \right) w^{(0)}\\
&= Q \, \Lambda^{-1} (D^{r+1} - I) \,\hat a.
\end{align*}
Taking the componentwise exponential and normalizing we get 
\[
W_i^{(r)} = \Big( \big(1+ \sum\limits_{\substack{l=1\\l \not = k}}^K G_{l,k} (r) \big)^{-1} \Big)_{k=1}^K,
\]
where
\[
G_{l,k}(r) \coloneqq \frac{\exp(q_i^\tT \Lambda^{-1} (D^{r+1} - I) \hat a_l)}{\exp(q_i^\tT \Lambda^{-1} (D^{r+1} - I) \hat a_k)}.
\]
In particular we have 
\begin{align*}
G_{l,k} (r)&= \prod_{j=1}^n \exp \left( q_{i,j} (\mu_j^{r+1} -1) \lambda_j^{-1} (\hat a_{j,l} - \hat a_{j,k}) \right),\\
g_{l,k} (r)&= \log( G_{l,k} (r) ) = \sum_{j=1}^n q_{i,j} (\mu_j^{r+1} -1) \lambda_j^{-1} (\hat a_{j,l} - \hat a_{j,k})\\
&= \sum_{s=1}^{M} c_{l,k}(s) (\mu_{n_s}^{r+1} -1).
\end{align*}
The final conclusions are analogous as in the proof of Theorem~\ref{thm:conv}.
\end{proof}
%-------------------------------------------------------------------------------
\section{Numerical Examples} \label{sec:numerics}
%-------------------------------------------------------------------------------
In this section we demonstrate the performance of Algorithm~\ref{alg:mult} for the partitioning of images. 
We start with color images having values in the Euclidean space and turn to manifold-valued images afterwards.
In particular we consider diffusion tensor magnetic resonance images (DT-MRI) 
with values in the Hadamard manifold of symmetric positive definite matrices
and electron backscatter diffraction (EBSD) images, whose values are taken from quotient spaces 
of the manifold of rotations $\mathrm{SO}(3)$. 

Let
${\mathcal G} \coloneqq \{1,\ldots,N\} \times \{1,\ldots,M\}$
be the image grid and
$\emptyset \not = {\mathcal V} \subseteq {\mathcal G}$
the set of available, in general noisy
pixels. In particular, we have ${\mathcal V} = {\mathcal G}$ in the case 
of no missing pixels.
We are interested in labeling an image
$f\colon {\mathcal V} \rightarrow {\mathcal M}$.

In our numerical examples we stop if the average entropy of the labeling matrix given by
\begin{equation}
- \frac{1}{\lvert\mathcal{G}\rvert}\sum_{i\in\mathcal{G}}\sum_{k=1}^K W_{i,k}^{(r)}\log\bigr(W_{i,k}^{(r)}\bigl),
\end{equation}
drops below a certain threshold. 
If not stated otherwise, we take $10^{-3}$ as threshold as in~\cite{APSS2016}, where the same stopping criterion was suggested. If not stated otherwise the same weights were used for $\alpha$ and $\rho$.

The algorithm was implemented in \textsc{MatLab} 2016a.

\paragraph{Color Images.}
\begin{figure}[tp]
	\centering
	\begin{subfigure}[b]{0.3\textwidth}
		\centering
		\includegraphics[width = 0.98\textwidth]{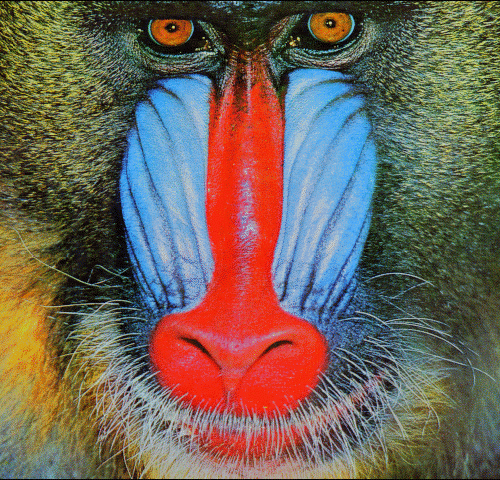}
		\caption[]{Original image \lstinline!mandrill!}\label{fig:mand:orig}
	\end{subfigure}
	\begin{subfigure}[b]{0.3\textwidth}
		\centering
		\includegraphics{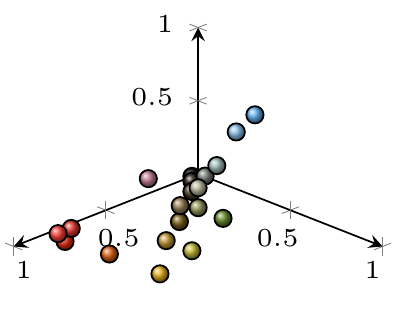}
		\caption[]{Prior features}\label{fig:mand:color}
	\end{subfigure}
	\\
	\begin{subfigure}{0.3\textwidth}
		\centering
		\includegraphics[width = 0.98\textwidth]{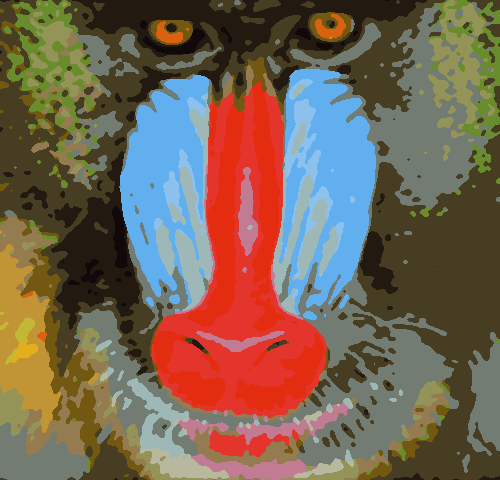}
		\caption[]{Local weights, $s = 3$}\label{fig:mand:without3}
	\end{subfigure}
	\begin{subfigure}{0.3\textwidth}
		\centering
		\includegraphics[width = 0.98\textwidth]{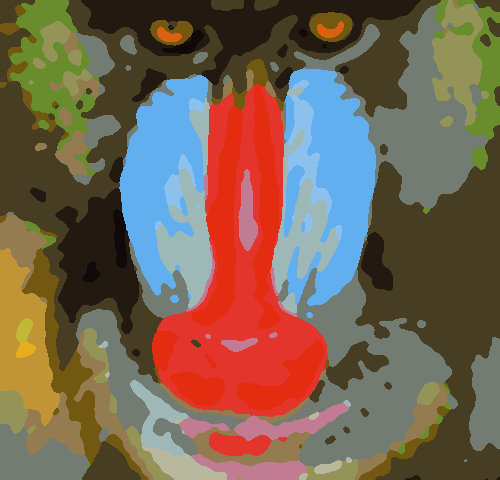}
		\caption[]{Local weights, $s = 5$}\label{fig:mand:without5}
	\end{subfigure}
	\begin{subfigure}{0.3\textwidth}
		\centering
		\includegraphics[width = 0.98\textwidth]{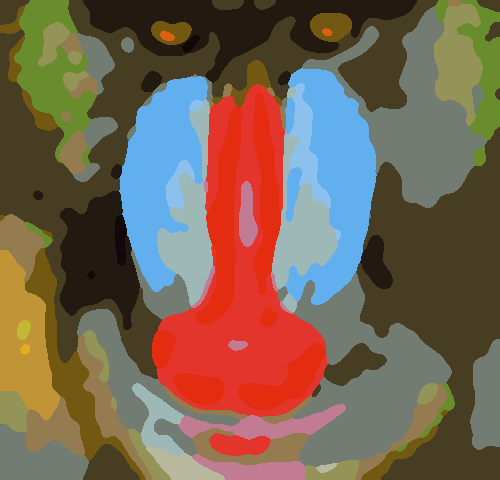}
		\caption[]{Local weights, $s = 7$}\label{fig:mand:without7}
	\end{subfigure}
	\\
	\begin{subfigure}{0.3\textwidth}
		\centering
		\includegraphics[width = .98\textwidth]{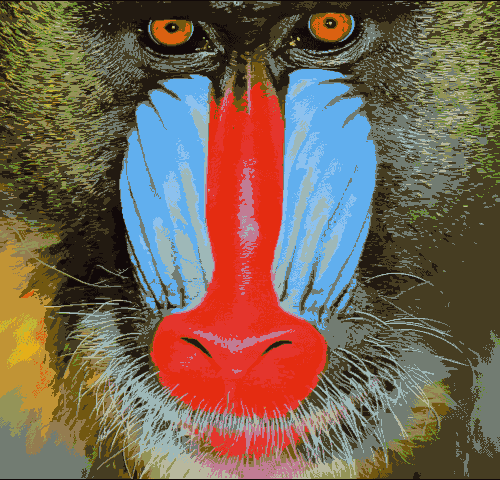}
		\caption[]{Nonlocal weights, \\$s_{\mathrm{nl}} =7$}\label{fig:mand:nl_ll}			
	\end{subfigure}
	\begin{subfigure}{0.3\textwidth}
		\centering
		\includegraphics[width = .98\textwidth]{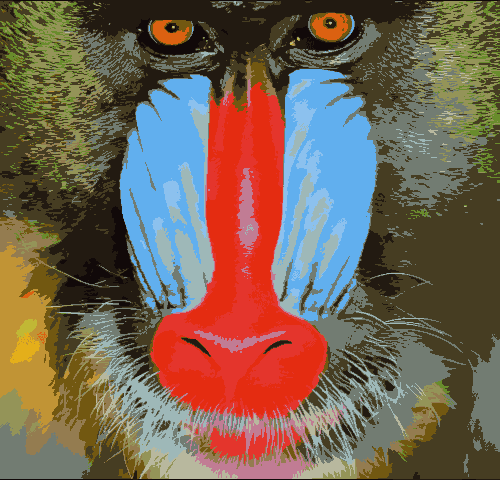}
		\caption[]{Nonlocal weights, $s_{\mathrm{nl}} =19$}\label{fig:mand:nl_l}			
	\end{subfigure}
	\begin{subfigure}{0.3\textwidth}
		\centering
		\includegraphics[width = .98\textwidth]{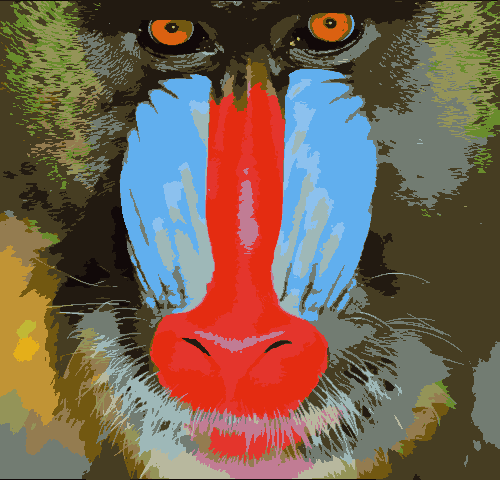}
		\caption[]{Nonlocal weights, $s_{\mathrm{nl}} =37$}\label{fig:mand:nl_ml}			
	\end{subfigure}
	\\
	\begin{subfigure}{0.3\textwidth}
		\centering
		\includegraphics[width = .98\textwidth]{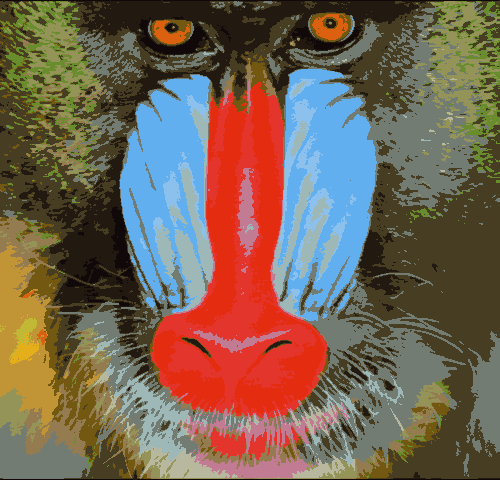}
		\caption[]{$\alpha$ from \subref{fig:mand:without3}, $\rho$ from \subref{fig:mand:nl_ll}}\label{fig:mand:ar_s}			
	\end{subfigure}
	\begin{subfigure}{0.3\textwidth}
		\centering
		\includegraphics[width = .98\textwidth]{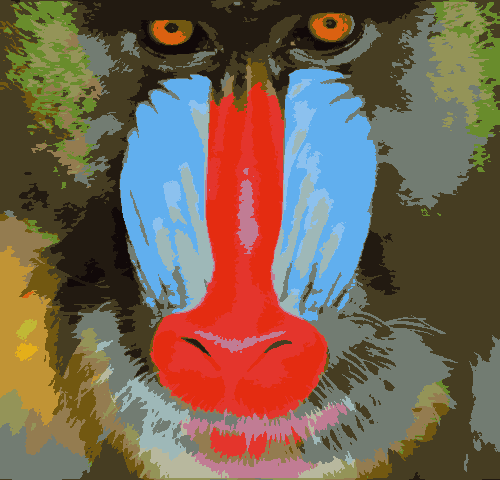}
		\caption[]{$\alpha$ from \subref{fig:mand:without5}, $\rho$ from \subref{fig:mand:nl_l}}\label{fig:mand:ar_m}			
	\end{subfigure}
	\begin{subfigure}{0.3\textwidth}
		\centering
		\includegraphics[width = .98\textwidth]{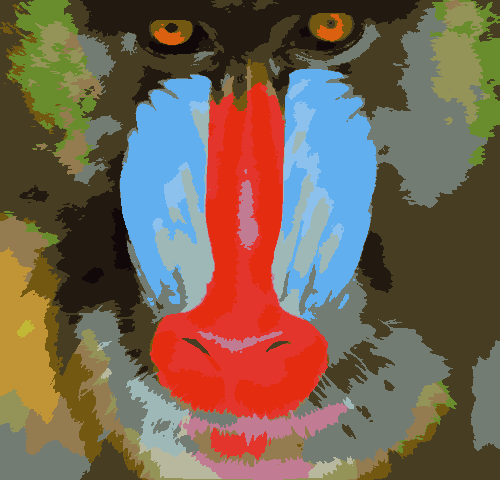}
		\caption[]{$\alpha$ from \subref{fig:mand:without7}, $\rho$ from \subref{fig:mand:nl_ml}}\label{fig:mand:ar_l}			
	\end{subfigure}
	\caption[]{Illustration of the influence of the chosen neighborhood in using Algorithm 1 for labeling the \lstinline!mandrill! with 20 different labels  
	and . 
	}\label{fig:mand}
\end{figure}
%--------------------------------------

In our \emph{first example} we label 
the \lstinline!mandrill! image $f\colon\mathcal{G}\rightarrow [0,1]^3 \subset \RR^3$ in ~Fig.~\ref{fig:mand:orig},
which was also handled in~\cite{APSS2016}.
We use the 20 color labels shown in Fig.~\ref{fig:mand:color} as prior features which were picked manually from the original image. 
We mention that the numerical results with Step 1 as in \eqref{step_1_ch}, i.e., with the
additional factor $A_i$ proposed in~\cite{APSS2016} differ only in 4 percent of the image pixels
from those presented below so that there is no visual difference.

First, we are interested in the influence of different weights $\alpha,\rho$, namely
\begin{enumerate}[label={\normalfont{\roman*)}}]
 \item \emph{uniform local} weights $\rho_{i,k} \coloneqq 1/s^2$, where $s^2 \coloneqq |{\mathcal N} (i)|$ is the size of the local $s \times s$ pixel neighborhood.
 \item \emph{nonlocal} weights which are constructed as follows:
a similarity measure for two pixels $i,j\in\mathcal{G}$ is given by the weighted difference of the surrounding $s_{\mathrm{p}} \times s_{\mathrm{p}}$ patches
\begin{equation}
\mathrm{d}_{\mathrm{p}}(i,j) \coloneqq  \sum_{l_1,l_2=-s_{\mathrm{p}}}^{s_{\mathrm{p}}} g_{\sigma_{\mathrm{p}}}(l_1,l_2) \distm(f_{i+(l_1,l_2)^\tT},f_{j+(l_1,l_2)^\tT}),
\end{equation}
where $g_{\sigma_{\mathrm{p}}}\colon\RR^2\rightarrow\RR$ 
is a Gaussian with standard deviation $\sigma_{\mathrm{p}}$. 
The neighborhood $\tilde{\mathcal{N}}(i)$ is then given by the $s_\mathrm{nl}$ pixels with smallest distances $\mathrm{d}_{\mathrm{p}}$ to $i$. 
Then we make the neighborhood symmetric by setting
\begin{equation}
	\mathcal{N}_\mathrm{nl}(i) \coloneqq \tilde{\mathcal{N}}(i)\cup\{j\in\mathcal{G}\vert i\in\tilde{\mathcal{N}}(j)\}.
\end{equation}
To reduce the computational effort for computing the weights we just compare each pixel with the pixels inside a $t\times t$ neighborhood.
Finally the weights $\rho_{i,j}$ are defined as 
\begin{align}
\rho_{i,j} \coloneqq \frac{\tilde{\rho}_{i,j}}{\sum_{j\in\mathcal{N}_\mathrm{nl}(i)}\tilde{\rho}_{i,j}}, \quad 
\tilde{\rho}_{i,j} \coloneqq \exp\biggl(-\frac{\mathrm{d}_{\mathrm{p}}(i,j)}{2\sigma^2_{\mathrm{w}}}\biggr).
\end{align}
\end{enumerate}
We mirror the image at the boundary. 

Fig.~\ref{fig:mand} shows the labeling with Algorithm~\ref{alg:mult}, more precisely:
\begin{itemize}
\item \emph{$\alpha$ local, $\rho$ local}:
In Figs.~\ref{fig:mand:without3}--\subref{fig:mand:without7} we used uniform local weights 
both for $\alpha$ and $\rho$ of different sizes.
We see that large neighborhoods  lead to coarse label assignments.
The number of iterations needed to reach the stopping criterion also increases with the size of the neighborhood from 27 to 39.
\item \emph{$\alpha$ nonlocal, $\rho$ nonlocal}:
Using nonlocal neighborhoods 
with parameters $s_{\mathrm{p}} = 3,\ s_{\mathrm{nl}} \in \{7,19,37\},\ t = 11,\ \sigma_{\mathrm{p}} = 1,\ \sigma_{\mathrm{w}} = 0.2$ both for $\alpha$ and $\rho$
results  in Figs.~\ref{fig:mand:nl_ll}--\subref{fig:mand:nl_ml}. 
Even though we take a similar number of neighbors in each column of Fig.~\ref{fig:mand}, we see 
that the nonlocal approach preserves the most details of the original image, such as the pupils or the whiskers. 
\item \emph{$\alpha$ local, $\rho$ nonlocal}:
In  Figs.~\ref{fig:mand:ar_s}--\subref{fig:mand:ar_l}, we look at the influence of choosing different $\alpha$ and $\rho$. 
Here we choose $\alpha$ as uniform local weights and $\rho$ as nonlocal weights from the respective column of Fig.~\ref{fig:mand}. 
The results are a mixture of the other two label assignments.
On the one hand we obtain larger connected regions as the purely nonlocal approach on the other hand we keep finer details such the pupils and 
nostrils in comparison to the local approach. 
\item \emph{$\alpha$ nonlocal, $\rho$ local}:
Switching $\alpha$ and $\rho$ leads to results very similar 
to those shown in Figs.~\ref{fig:mand:without3}--\subref{fig:mand:without7}.
\end{itemize}

\begin{figure}[tp]
	\centering	
	\begin{subfigure}[t]{0.32\textwidth}
		\centering
		\includegraphics{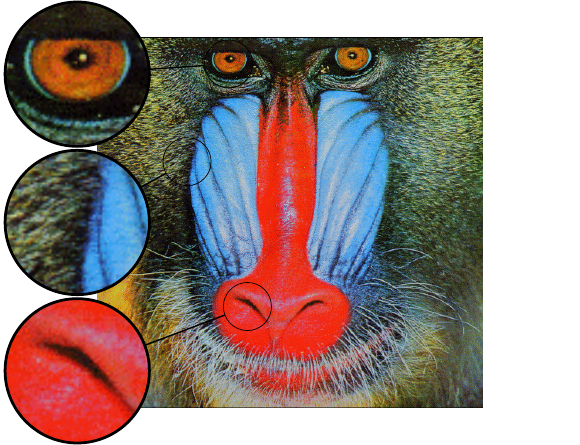}
		\caption[]{Original image \lstinline!mandrill!}\label{fig:comp:orig}
	\end{subfigure}
  \\
	\begin{subfigure}[t]{0.32\textwidth}
		\centering
		\includegraphics{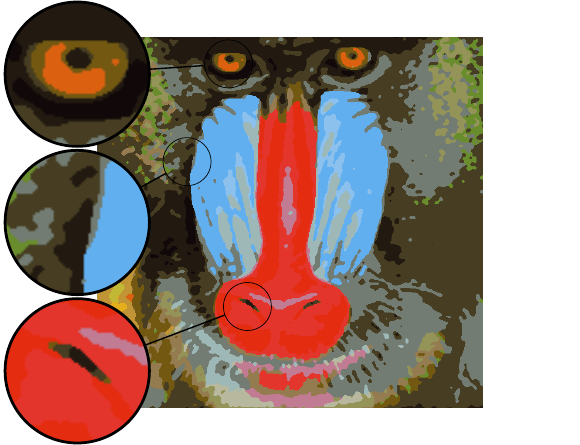}
		\caption[]{Our method with local weights,
    \(s=3\), cf.~Fig.~\ref{fig:mand:without3}}
    \label{fig:comp:without3}
	\end{subfigure}	
	\begin{subfigure}[t]{0.32\textwidth}
		\centering
		\includegraphics{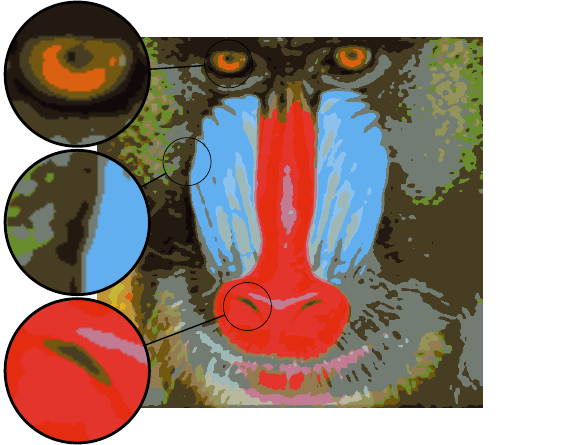}
		\caption[]{Simple Labeling, $\sigma = 2.5$}\label{fig:comp:gauss25}
	\end{subfigure}	
	\begin{subfigure}[t]{0.32\textwidth}
		\centering
		\includegraphics{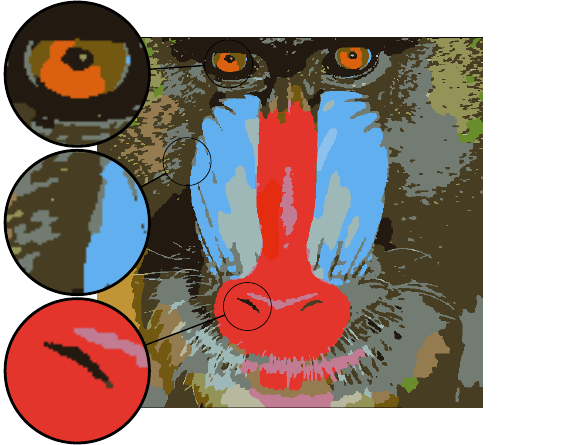}
		\caption[]{Method from~\cite{CCP2012}, $\lambda=15$}\label{fig:comp:pock}
	\end{subfigure}	
	\caption[]{Comparison of~\subref{fig:comp:without3} our Algorithm~\ref{alg:mult} with two other labeling methods: \subref{fig:comp:gauss25} a Simple Labeling and \subref{fig:comp:pock} a variational model from~\cite{CCP2012} using the labels from Fig.~\ref{fig:mand:color}.}\label{fig:comp}
\end{figure}
Next we want to compare the proposed algorithm with two other labeling methods.
The first method, called \emph{Simple Labeling} in the following, starts by computing the convolution of each of the RGB channels with a Gaussian of standard deviation \(\sigma\). The label assignment is then given by choosing the closest prior feature at each pixel, i.e., the label of pixel $i$ is given by $\argmin_k \distm(f_{\sigma,i},f^*_k)$, where $f_\sigma$ is the image $f$ convolved with the Gaussian.
Note that for more general manifold-valued images as considered in the next
examples, the Riemannian center of mass would be the natural choice for
averaging, but its computation involves a minimization process
and can not be stated explicitly.
The second method we compare with is a variational method given in~\cite{CCP2012}. Of course there exists many more methods to compare to; we chose these two, because the first one is even simpler, and the second one is a state-of-the-art variational model for image partitioning.

In Fig.~\ref{fig:comp} we compare these two methods with the multiplicative labeling derived in this paper. In order to compare the methods, we choose the parameters \(s, \sigma, \lambda\) from our approach, the Simple Labeling and the method from~\cite{CCP2012}, respectively, such that they result in similar labeling in the background, e.g., the fur in the middle right part of the image of the \lstinline!mandrill!. We further focus on three regions: the left eye, the left side of the nose and the left nostril. The original image, shown in Fig.~\ref{fig:comp:orig}, is labeled using the prior features from the last example, cf.~Fig.~\ref{fig:mand:color}.

When we compare our method with local weights and \(s=3\), cf.~Fig.~\ref{fig:mand:without3}, whose focus regions are shown in Fig.~\ref{fig:mand:without3} to the result of the Simple Labeling, cf.~Fig.~\ref{fig:comp:gauss25}, \(\sigma=2.5\), the details are kept sharper in our method, while they get blurred in the latter approach. This can clearly be seen at the nostril or the nose detail, where at least one additional label is introduced at the boundary, and the pupil is kept in its original size in our approach. However, in e.g.~more homogeneous regions like the red upper part of the nose or the fur in the top right, both methods result in similar label assignments.

Compared to the variational method from~\cite{CCP2012} our approach seems to keep a little more detail e.g. at the boundary of the nose or the fur in the bottom right,
while just loosing a little bit of detail, like the white spot in the eye. Despite of these two minor differences they perform similar.
%---------------------------------------------
\begin{figure}
	\centering
	\begin{subfigure}{0.4\textwidth}
		\centering
		\includegraphics[width = 0.98\textwidth]{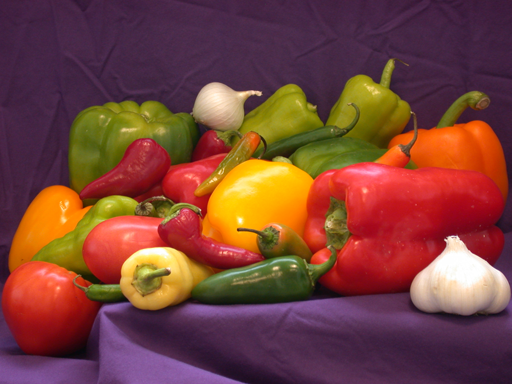}
		\caption[]{Original image \lstinline!peppers!}\label{fig:pepper:orig}			
	\end{subfigure}
	\begin{subfigure}{0.4\textwidth}
		\centering
		\includegraphics[width = 0.98\textwidth]{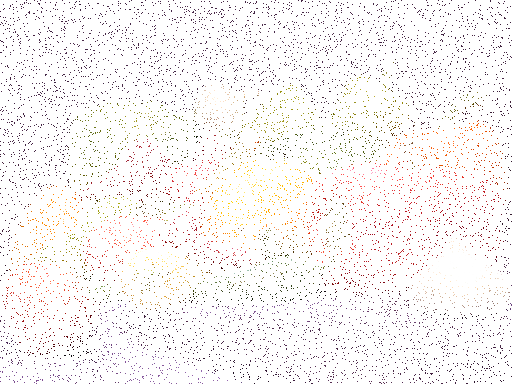}
		\caption[]{Lossy image, 95\% of missing pixels}\label{fig:pepper:lossy}			
	\end{subfigure}
	
	\begin{subfigure}{0.4\textwidth}
		\centering
		\includegraphics[width = 0.98\textwidth]{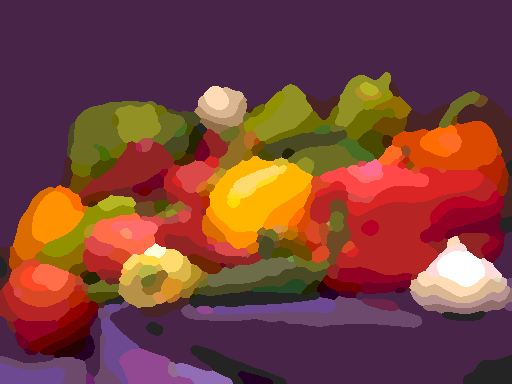}
		\caption[]{Labeling of the original image}\label{fig:pepper:lab}			
	\end{subfigure}
	\begin{subfigure}{0.4\textwidth}
		\centering
		\includegraphics[width = 0.98\textwidth]{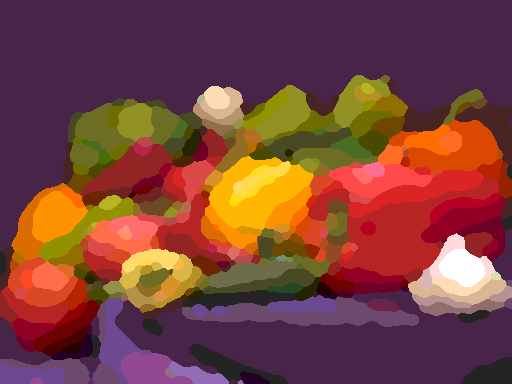}
		\caption[]{Labeling of the lossy image}\label{fig:pepper:inp}			
	\end{subfigure}
	\caption[]{Comparison between labeling the original image~\subref{fig:pepper:orig} and the lossy one \subref{fig:pepper:lossy}
	for uniform local weights with $s = 5$.}
\end{figure}

In our \emph{second example} we apply the proposed Algorithm~\ref{alg:mult} using local weights as in i) 
in order to label  the \lstinline!pepper! image, see Fig.~\ref{fig:pepper:orig}. At this point we assume that we are only 
given data at grid points $\mathcal{V}\subset\mathcal{G}$, $\lvert\mathcal{V}\rvert\ll\lvert\mathcal{G}\rvert$,
where 95\% of the pixel values are unknown, cf. Fig.~\ref{fig:pepper:lossy}. 
The prior features are chosen by a discretization of the color cube, more precisely we consider, 
$\bigl\{(f^*_1,f^*_2,f^*_3): f^*_i\in \{0,\tfrac{1}{7},\dots,1\},\ i\in\{1,2,3\}\bigr\}$. 
The distance matrix ${\mathcal D} \coloneqq (d_{i,k})_{i=1,k=1}^{n,K}$ is given by
\begin{align}
d_{i,k} &\coloneqq\begin{cases}
\distm(f_i,f_k^*)\quad & i\in\mathcal{V},\\
0\quad & i\notin\mathcal{V}.
\end{cases}
\end{align} 
Fig.~\ref{fig:pepper:inp} shows the resulting labeling. Comparing this with the labeling of the original image shown in Fig.~\ref{fig:pepper:lab}
we see that our approach can cope with missing data very well.

%-------------------------------------------------------------------------------
\paragraph{Symmetric Positive Definite Matrices.}
We want to label images having values on 
the manifold of symmetric positive definite $r \times r$ matrices ${\cal P}(r)$ with the 
affine invariant distance, see e.g.~\cite{AFPA2006,pennec2006riemannian},
\[
	\distm_{{\mathcal P}(r)} (x_1,x_2) \coloneqq \bigr\lVert\MLog(x_1^{-\frac12} x_2 x_1^{-\frac12} )\bigl\rVert,
\]
where \( \MLog x \coloneqq -\sum_{k=1}^\infty \frac{1}{k} (I-x)^k, \rho(I-x) < 1,
\) denotes the matrix logarithm of \(x\in\mathcal P(r)\) and the norm is the Frobenius norm.

As a \emph{first example} we look at diffusion tensors occurring in magnetic resonance imaging (DT-MRI). 
For DT-MRI a 3D dataset from MRI is used to compute diffusion tensors at each measurement position resulting in a
dataset, where each data point is a symmetric positive definite matrix. 
Hence all data items are given on the manifold \(\mathcal P(3)\) of symmetric positive definite \(3\times 3\) matrices. 
The Camino project\footnote{see \url{http://camino.cs.ucl.ac.uk}}~\cite{camino} provides a freely available dataset of DT-MRI of the human head. 
We take a look at the traversal cut at \(z=35\) in the dataset, i.e., our data is a manifold-valued image of size \(N\times M\) with \(N=M=112\)
where the pixels outside the head are missing, see Fig.~\ref{subfig:DT-MRI:ori}. 
The coloring is based on the anisotropy index, cf.~\cite{MoBa06}, employing the Matlab colormap \lstinline!hsv!.

We ignore the outer pixels in the following by setting their distance (and hence influence) to any other pixel to zero.
We define two prior features, which are depicted as pink ellipses, one denoting the inner, less diffusion affected area, one to capture the boundary and main diffusion areas of the brain. 
We compare three approaches, i) an equally weighted neighborhood of size \(s\times s\), ii) a weighted \(3\times 3\) using a truncated Gaussian, \(\sigma=0.5\) 
and iii) the nonlocal weights. We set \(s_{\mathrm{p}} = 3\), \(s_{\mathrm{nl}} = 9\), \(t=15\),
and for the two involved Gaussians \(\sigma_{\mathrm{p}} = 3\) and \(\sigma_{\mathrm{w}} = \frac{1}{5}\). 
The results are shown in Figs.~\ref{fig:DT-MRI}\,\subref{subfig:DT-MRI:eq}--\subref{subfig:DT-MRI:nl}:
while the equally weighted neighborhood captures the boundary of the brain correctly despite a small part at the top right, it is not able to “follow” the
diffusivity paths inside the brain. Using the weighted neighborhood captures the main features while loosing a little detail on the top right border. 
The nonlocal approach captures the complete boundary and even more details of diffusion paths inside the brain.
\begin{figure}[!tbp]\centering
	\begin{subfigure}[t]{.32\textwidth}
		\centering
		\includegraphics[width=.95\textwidth]{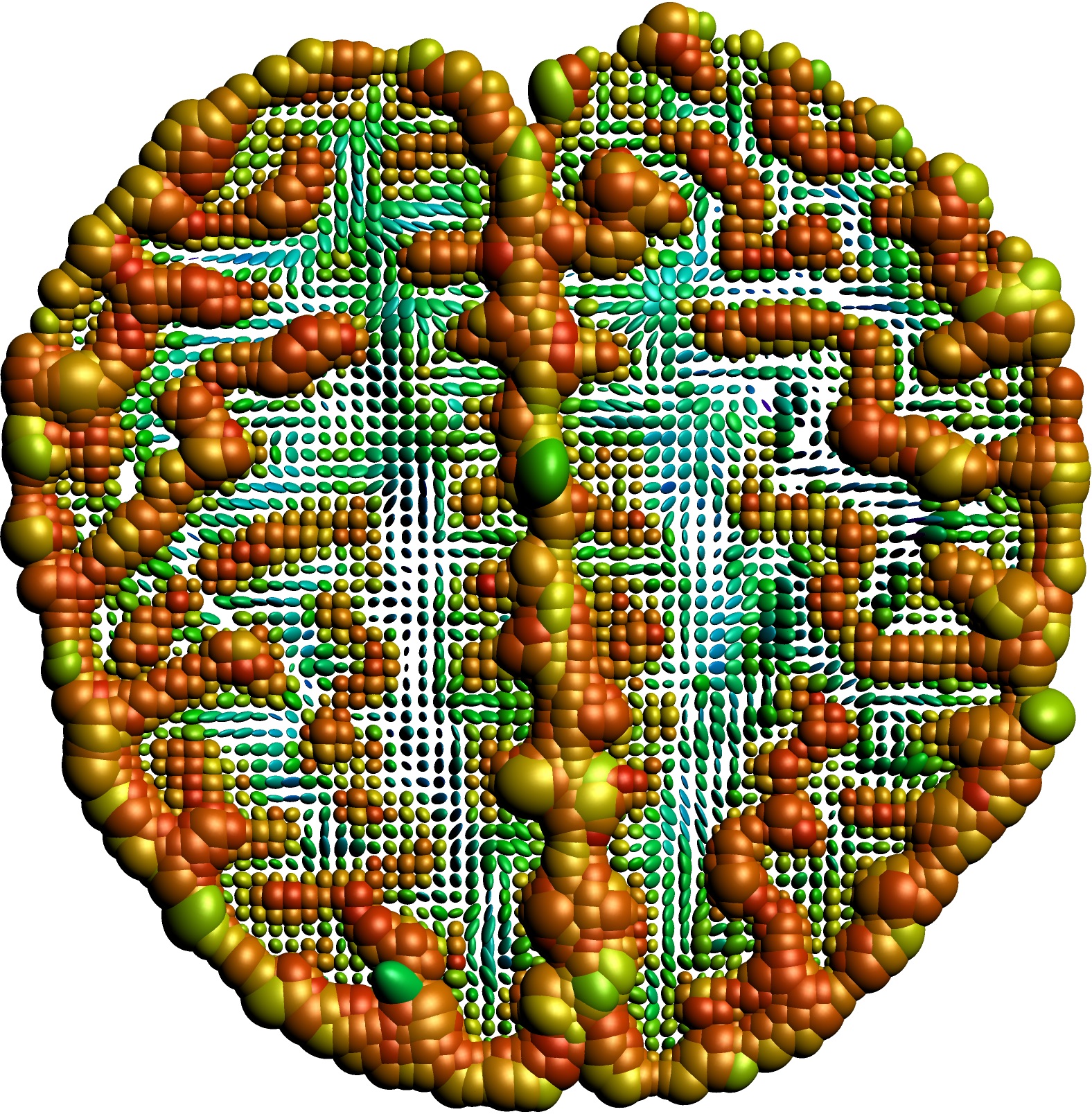}
		\caption{Original DT-MR image}\label{subfig:DT-MRI:ori}
	\end{subfigure}
	\\
	\begin{subfigure}[t]{.32\textwidth}
		\centering
		\includegraphics[width=.95\textwidth]{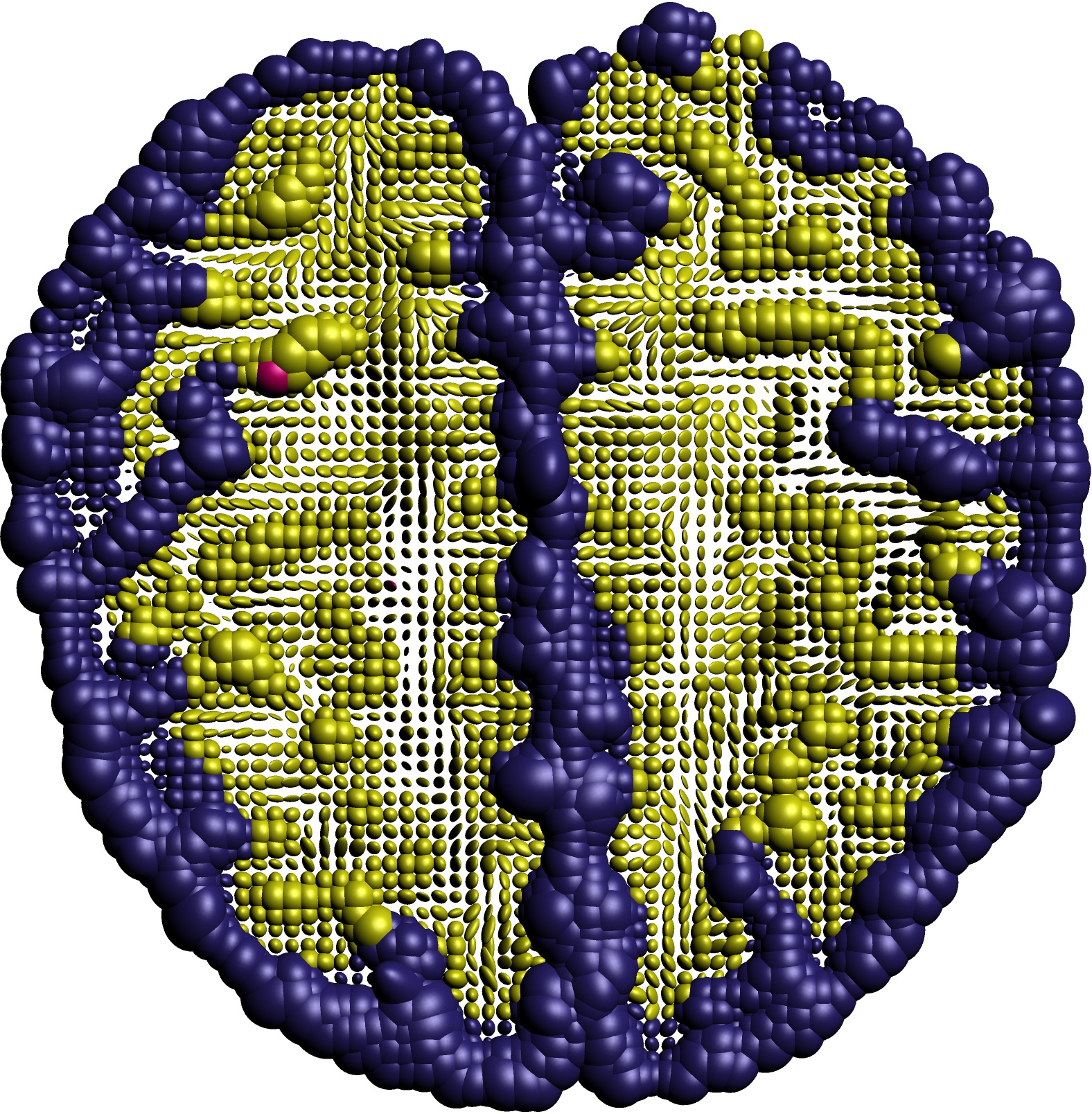}
		\caption{Equal weights, \(s=3\)}
		\label{subfig:DT-MRI:eq}
	\end{subfigure}
	\begin{subfigure}[t]{.32\textwidth}
		\centering
		\includegraphics[width=.95\textwidth]{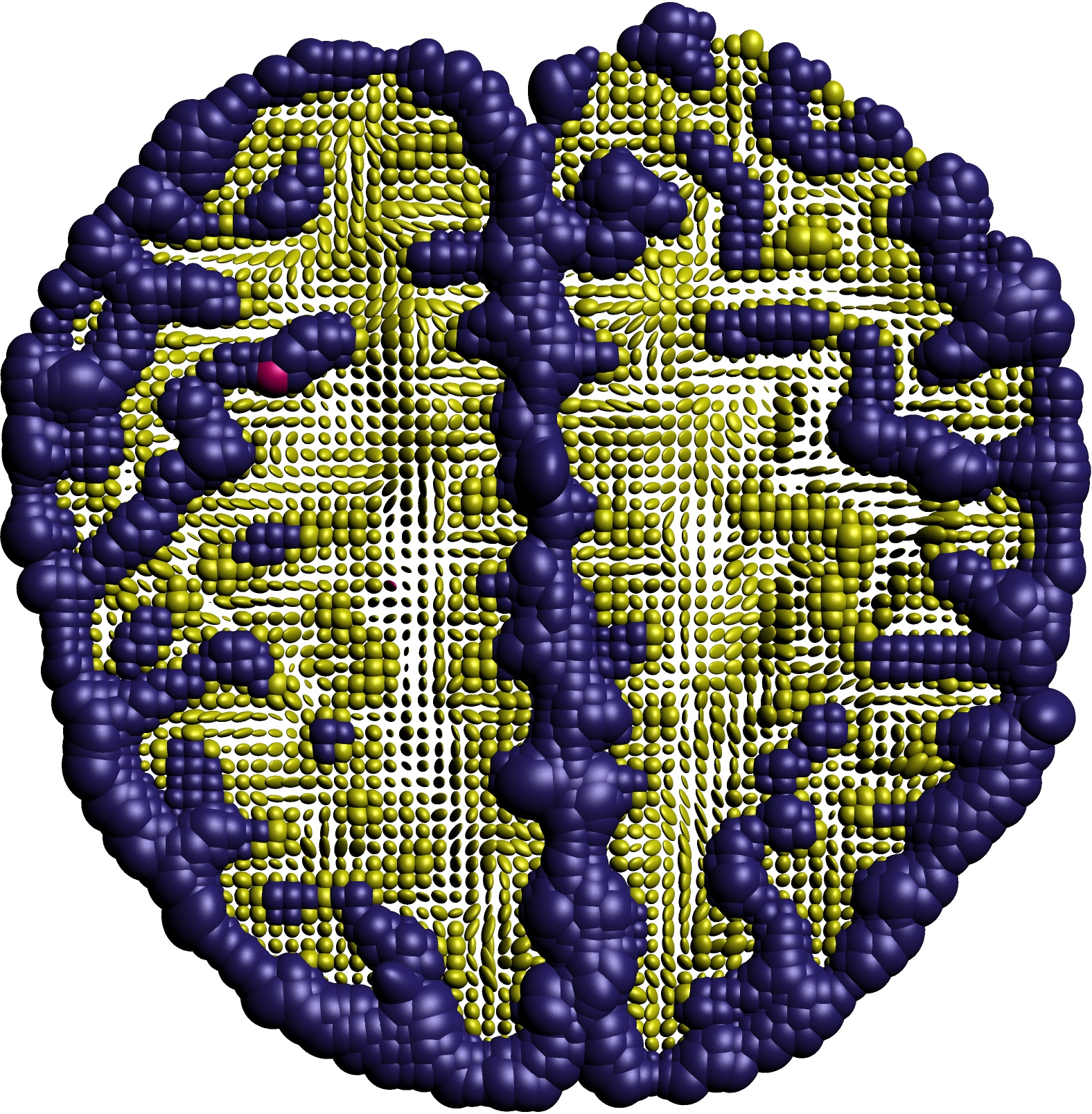}
		\caption{Gaussian weights on \(3\times 3\) neighborhood, \(\sigma=0.5\)}
		\label{subfig:DT-MRI:w}
	\end{subfigure}
	\begin{subfigure}[t]{.32\textwidth}
		\centering
		\includegraphics[width=.95\textwidth]{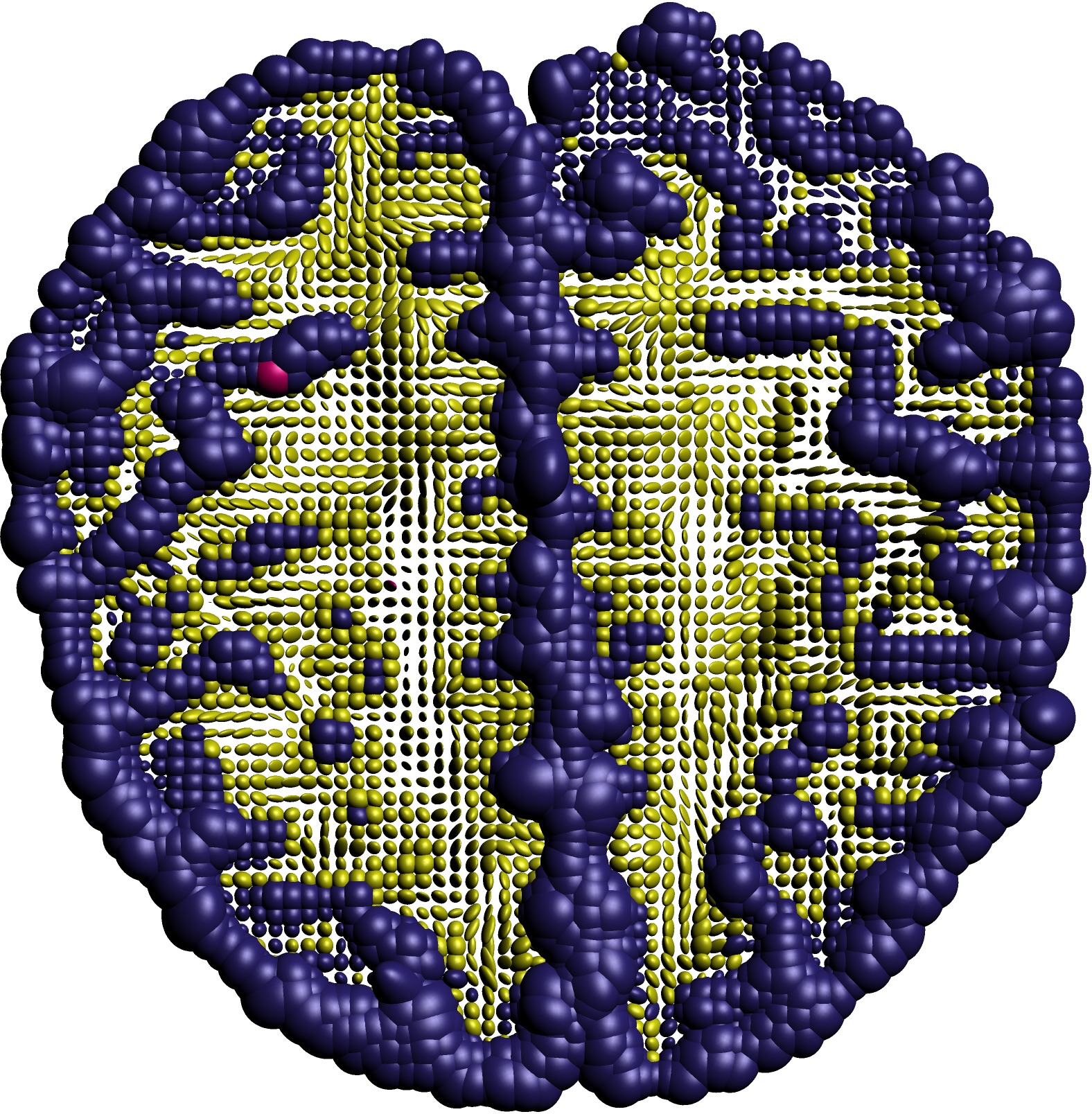}
		\caption{Nonlocal weights}
		\label{subfig:DT-MRI:nl}
	\end{subfigure}
	\caption{Segmentation of DT-MRI data with two labels and different neighborhoods.}
	\label{fig:DT-MRI}
\end{figure}
\medskip

In our \emph{second example} for symmetric positive definite matrices we have a look at the texture segmentation approach from \cite{TPM2006}. 
There the authors assign to each pixel of the image a covariance matrix, which is symmetric positive semidefinite by definition, 
and we can assume it is positive definite for a large enough number of samples in natural images. These covariance matrices are based on the following feature vectors
\begin{equation}
F(i) = \bigl(I(i),I_x(i),I_y(i),I_{xx}(i),I_{yy}(i)\bigr)^\tT,
\end{equation}
where $I$ is the intensity of the image, $I_x,I_{xx}$, and $I_y,I_{yy}$, are the first and second derivatives in horizontal, respectively vertical direction. 
To avoid the influence of noise we convolve the image with a Gaussian of small standard deviation $\sigma$ before computing the feature vectors. 
The derivatives are discretized via central differences. Given a neighborhood $\mathcal{N}_{\textrm{cov}}(i)$ around each pixel the covariance matrices are defined as 
\begin{equation}
C(i) = \frac{1}{\lvert\mathcal{N}_{\textrm{cov}}(i)\rvert-1}\sum_{j\in\mathcal{N}_{\textrm{cov}}(i)} \Bigl(F(j)-\mu(i)\Bigr)\Bigl(F(j)-\mu(i)\Bigr)^\tT,
\end{equation}
where $\mu(i)\in\RR^5$ is the mean of the features $F$ in the neighborhood $\mathcal{N}_{\textrm{cov}}(i)$. 
To achieve a good segmentation result we also include the mean $\mu$ into our final features. So we work on the product manifold $\RR^5\times\SPD(5)$ with distance
\begin{equation}
\distm_{\RR^5\times\SPD(5)}\bigl((\mu_1,C_1),(\mu_2,C_2)\bigr) 
= \sqrt{\lVert\mu_1-\mu_2\rVert_2^2+\distm_{\SPD(5)}^2(C_1,C_2)}.
\end{equation}
If we want to segment the texture image in Fig.~\ref{fig:intro:orig} we need appropriate prior features. 
They are obtained from six supervising texture images using 100 randomly selected rectangular patches 
and calculate the mean of their means and covariance matrices. The result in Fig.~\ref{fig:intro:res} 
is then obtained by using $\sigma= 0.5$ for the convolution, $7\times7$ patches as neighborhoods $\mathcal{N}_{\textrm{cov}}(i)$ and $15\times15$ 
patches as  neighborhoods $\mathcal{N}_\rho$ with uniform weights in Algorithm~\ref{alg:mult}. 
%
%----------------------------------------------------------------------------------------------------
\paragraph{EBSD Data.}
The analysis of polycrystalline materials by means of electron backscattered diffraction
(EBSD)  is a fundamental task in material science, see~\cite{ASWK93,KuWrAdDi93}. 
Since the microscopic grain structure
affects macroscopic attributes of materials such as ductility, electrical and
lifetime properties, there is a growing interest in the grain structure of
crystalline materials such as metals and minerals. 
EBSD provides for each
position on the surface of a specimen a so-called Kikuchi pattern, which
allows the identification of 
\begin{itemize}
  \item[i)] the structure (material index), and
  \item[ii)] the orientation of the crystal relative to a fixed coordinate
system ($\operatorname{SO}(3)$ value).
\end{itemize}
Since the atomic structure of a crystal is invariant
under its specific finite symmetry group $S \subset \operatorname{SO}(3)$, 
the orientation at each grid point $i \in {\mathcal G}$ is only given as an equivalence class 
\[
	[f_i] = \{f_i s \mid s \in S\} \in \operatorname{SO}(3)/S, \quad f_i \in \operatorname{SO}(3).
\]

The following remark introduces the colorization method which is implemented in the software package MTEX~\cite{MTEX,NH16} and which we will use for visualization.

\begin{remark} 
The colorization of $\operatorname{SO}(3)/S$ valued rotations $[f]$ is done as follows: for a
fixed vector $\vec{r} \in \mathbb S^2$ we consider the mapping
$\Phi \colon \operatorname{SO}(3)/S \to \mathbb S^{2}/S$, $[f] \mapsto [f^{-1} \vec r]$ and then we assign a certain color to each point on $S^{2}/S$. 
Usually $\vec{r}$ is defined as the vector orthogonal to the specimen surface. Note that this reduces the dimension by one since we can no longer distinguish between rotations around $\vec{r}$.
After applying $\Phi$ we introduce a color coding of the sphere that takes into account the symmetries.

To illustrate the color coding, we choose quartz as an example as it has neither trivial nor too complicated symmetries.
The colorization of the quotient $\mathbb S^{2}/S$ for this example is depicted in Fig.~\ref{ebsdraw}. 
The color map shown in Fig.~\ref{ebsdraw} c) respects the symmetry group
$S \subset \operatorname{SO}(3)$ of quartz. The symmetry group has three rotations with
respect to a 3-folded axis ($2k \pi/3 $, $k=1,\ldots,3$ rotations around the main axis $c$ of the crystal) and three rotations with respect to 2-folded axis, e.g., $a_{1}$, $a_{2}$,
perpendicular to $c$.
The polyhedron drawn in Fig.~\ref{ebsdraw} a) and b) to visualize the orientation has exactly these symmetries.
\end{remark}

% ----------------------.................--------------------------------------
\begin{figure}
	\centering
	\begin{subfigure}[b]{0.3\textwidth}
        \includegraphics{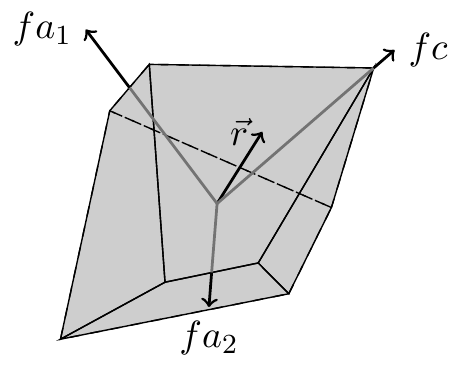}
			\subcaption{Visualization of the crystal orientation $f$}	
		\end{subfigure}
		\begin{subfigure}[b]{0.3\textwidth}
        \includegraphics{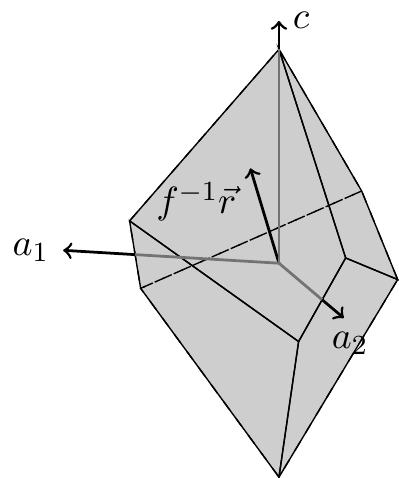}
			\subcaption{Crystal rotated back to reference system}	
		\end{subfigure}
		\begin{subfigure}[b]{0.35\textwidth}
		\centering
		\includegraphics[width=.7\textwidth]{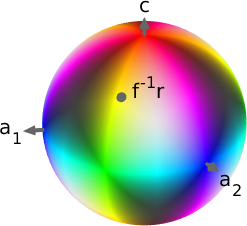}
		\subcaption{Color map for the symmetry group of quartz}
	\end{subfigure}
	\caption[]{
	 Visualization of the rotations considering the symmetry group of quartz  
	according to the mapping $\operatorname{SO}(3)/S \ni [f] \mapsto [f^{-1} r] \in \mathbb S^{2} /S$.
	}\label{ebsdraw}
\end{figure}
The grains in EBSD images are regions with similar orientations which we want to label.
As distance between orientations on $\operatorname{SO}(3)$ we use the geodesic distance 
\begin{equation}
\distm_{\operatorname{SO}(3)}(f_1,  f_2)\coloneqq \arccos\lvert\langle f_1,f_2\rangle\rvert,
\end{equation}
where we suppose that the orientations are given as quaternions, see~\cite{graef12}, i.e., $f_i\in\mathbb{S}^3/\{-1,1\}$.
To cope with the symmetry, we apply the usual distance on the quotient space $\operatorname{SO}(3)/S$, cf.~\cite[p. 153]{Fol1999}:
\begin{align}
	\distm_{\operatorname{SO}(3)/S}(f_1,f_2) \coloneqq \min_{\tilde f_1 \in [f_1]_S, \tilde f_2 \in [f_2]_S} \distm_{\operatorname{SO}(3)}(\tilde f_1, \tilde f_2).
\end{align}
In EBSD data usually pixels are missing, e.g., due to corrupted Kikuchi patterns.
Therefore we extend the approach to this setting and modify the distance as in the previous paragraph:
if a pixel is missing, the distance between this pixel and every cluster center is set to $0$.
As prior features we randomly choose orientations in the original data 
in such a way that the distance between all prior features is at least $0.3$ and 
for each pixel in the original data there exists a prior feature with distance less or equal to $0.3$.

In our \emph{first example} we label EBSD data obtained from the single phase material Magnesium, where few pixels are missing, see~\cite{MTEX}.
We focus on Algorithm~\ref{alg:mult}.
We fix $\nu=1$ and choose Gaussian weights 
$\alpha_{i,k} = \rho_{i,k}$ with standard deviation $\sigma = 0.8$ which are truncated  outside $[-3 \sigma, 3\sigma]$,
i.e., we work in a local $5\times 5$ neighborhood. 
Again, we mirror the image at the boundary.
The result is shown in Fig.~\ref{fig:ebsd_1phase}.
In contrast to the method for finding grain boundaries via thresholding described in~\cite{BHS11} 
and implemented in the software package MTEX~\cite{MTEX} 
the proposed approach yields smoother grain boundaries and is able to reduce noise.
The grain boundaries are sometimes too fringy in the original data.
This happens since the diffraction patterns are inaccurate at grain boundaries due to the change in orientation and therefore also the assigned rotations are inaccurate.
Here, the labeling method performs better.

\begin{figure}\centering
	\begin{subfigure}[b]{0.3\textwidth}\centering
	\includegraphics[height=1.2\textwidth]{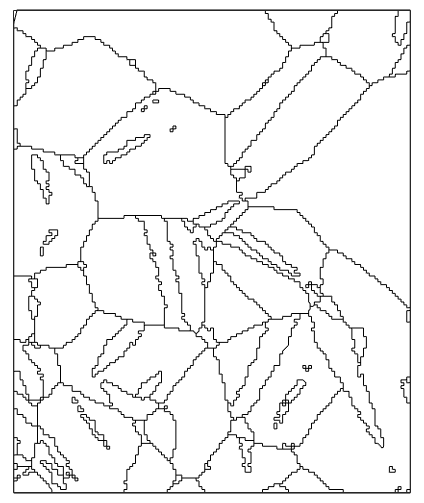}
	\caption[]{Grain boundaries computed with MTEX.}
	\end{subfigure}
	\begin{subfigure}[b]{0.3\textwidth}\centering
	\includegraphics[height=1.2\textwidth]{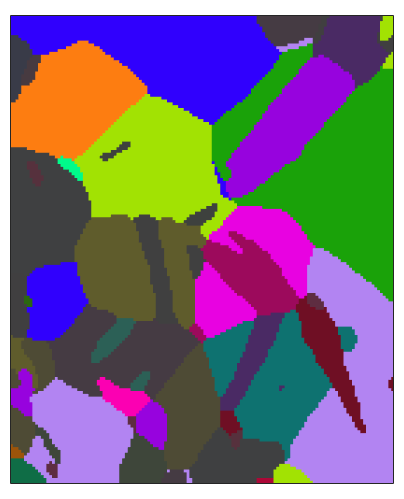}
	\caption[]{Labeling with our approach.}
	\end{subfigure}
	\caption[]{Labeling of single phase EBSD data (Magnesium) with $31$ prior features.
	$\mathcal{N}$ is chosen as the $5\times5$ neighborhood with truncated Gaussian weights ($\sigma = 0.8$). The proposed model is able to reduce noise and produces smoother boundaries.
	\label{fig:ebsd_1phase}}
\end{figure}

In the \emph{second example} we draw our attention to multiphase EBSD data.
We consider a multiphase dataset consisting of the three phases Forsterite, Enstatite and Diopside taken from MTEX~\cite{MTEX}.
These phases have different crystal symmetries such that we can not directly compute the distance between orientations belonging to different phases.
Therefore, the features have to represent not only the orientations, but also the phase they are corresponding to.
We consider features $(f_i,p_i)$ where $p_i \in \{1,2,3\}$ denotes the phase and  $f_i \in \operatorname{SO}(3)/S(p_i)$ is the orientation where $S(p_i)$ 
denotes the symmetry group of the phase $p_i$.
Then we choose a finite, but constant distance $\tau$ between pixels that belong to different phases to allow for phase changes of noisy pixels.
Now, the distance matrix ${\mathcal D} \coloneqq (d_{i,k})_{i=1k=1}^{n,K}$ has entries
\begin{align}
d_{i,k} &\coloneqq \begin{cases} 0 \quad& i\notin\mathcal{V}, \\
\distm_{\operatorname{SO}(3)/S(p_i)}(f_i,f_k^*) &  i\in\mathcal{V},\ p_i = p_k^*,\\
\tau & \text{otherwise.}
\end{cases}
\end{align}
Fig.~\ref{fig:ebsd_multiphase} shows an example of such a data set. Here we set $\tau \coloneqq 1$.
We see that the partitioning of both, phase and orientation works well even in the presence of missing (white) pixels.
In particular, also the phase boundaries get smoothed.
Furthermore, noise in the phase data, i.e., pixels that were assigned to the wrong phase in the image acquisition process, is reduced.

%----------------------------------------------------
\begin{figure}\centering
	\begin{subfigure}[t]{0.23\textwidth}\centering
	\includegraphics[height=3.3\textwidth]{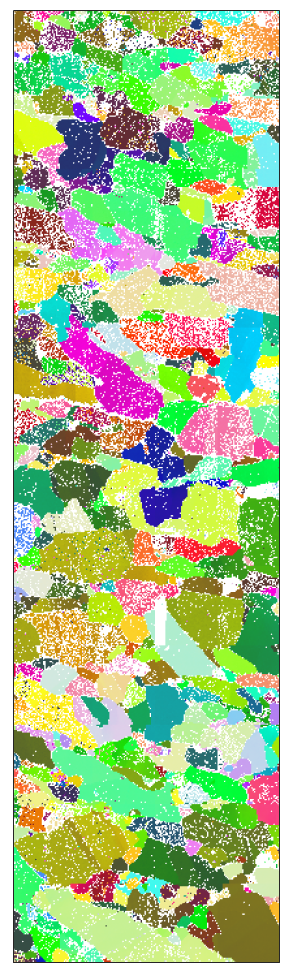}
	\caption[]{Original data.}
	\end{subfigure}
	\begin{subfigure}[t]{0.23\textwidth}\centering
	\includegraphics[height=3.3\textwidth]{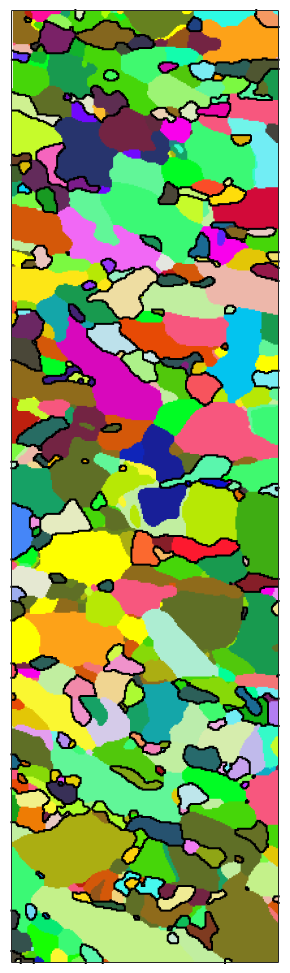}
	\caption[]{Labeled orientations.}
	\end{subfigure}
	\begin{subfigure}[t]{0.23\textwidth}\centering
	\includegraphics[height=3.3\textwidth]{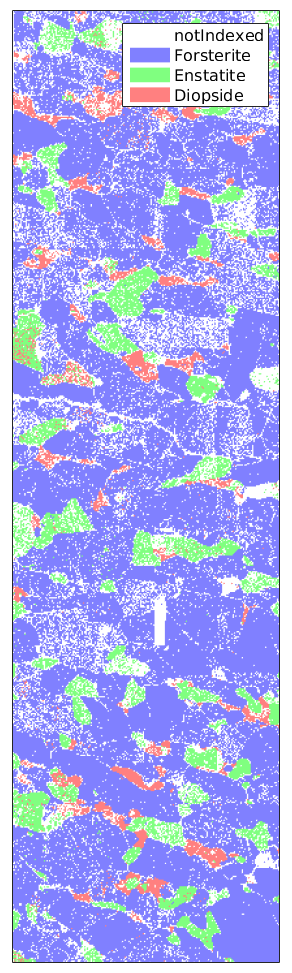}
	\caption[]{Original phases.}
	\end{subfigure}
	\begin{subfigure}[t]{0.23\textwidth}\centering
	\includegraphics[height=3.3\textwidth]{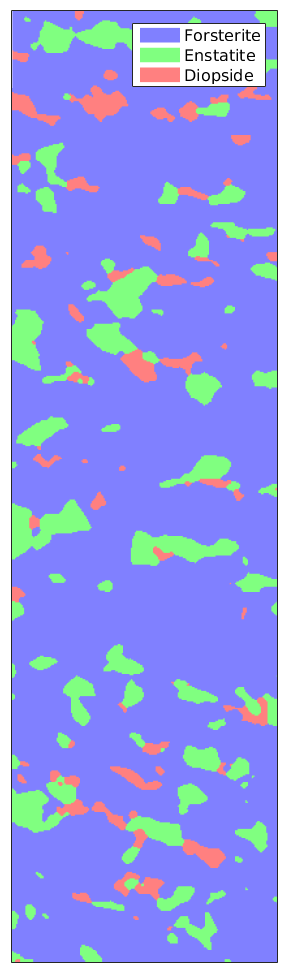}
	\caption[]{Labeled phases.}
	\end{subfigure}
	\caption[]{Labeling of corrupted three-phase EBSD data with $447$ prior orientations
	by Algorithm~\ref{alg:mult}. 
	For the original data also the MTEX colorization was applied.
	Phase transitions are emphasized by black lines in the labeled orientations.
  }	\label{fig:ebsd_multiphase}
\end{figure}
%-------------------------------------------------------------------------------
\section{Conclusions} \label{sec:conclusions}
%-------------------------------------------------------------------------------
We have proposed a simple and efficient algorithm
for data labeling based on an iterative geometric averaging 
of a label assignment matrix.
The algorithm starts with a matrix  of averaged distances between the prior features and 
those which we want to label.
In applications this distance may not be the Euclidean one, but e.g., the distance on an
appropriate manifold.
Since  the weights in the label assignment matrix are updated using only few of their local or nonlocal neighbors, 
our algorithm has a high potential for parallelization, which is not exploited in the implementation so far.
In particular, a large number of labels does not increase the convergence speed drastically, which is the case for various
TV-regularized variational methods as~\cite{BYT2011,CCP2012,CJPT11,CEN2006,LS2011}.

There are still a lot of open questions, which we want to examine in our future work, e.g.:
\begin{enumerate}[label={\roman*)}]
 \item
the convergence analysis of Algorithm~\ref{alg:mult} with Step 3, in particular the role
of $\varepsilon$. Note that we have observed convergence to the vertices of $\Delta_{K,\varepsilon}$
in all our numerical examples.
\item and to the solution of variational problems.
\end{enumerate}
Of course it would be useful to have a similar multiplicative filtering approach for non supervised labeling 
or to develop updating strategies for the prior features. Further, we intend to use the method for other
applications.

%-------------------------------------------------------------------------------
\paragraph{Acknowledgement.} 
 We are grateful to Ch. Schn\"orr (University of Heidelberg) for stimulating discussions.
Many thanks to R. Hielscher (University of Chemnitz) for supporting the work on EBSD data.
We thank the referees for requesting a discussion of condition (PI).
Funding by the German Research Foundation (DFG) within the project STE 571/13-1 \& BE 5888/2-1
and with\-in the Research Training Group 1932 “Stochastic Models for Innovations in the Engineering Sciences”, 
project area P3, is gratefully acknowledged.
Furthermore, G. Steidl acknowledges the support by the German Federal Ministry of Education and Research (BMBF) through grant 05M13UKA
(AniS).

\bibliographystyle{abbrv}
\bibliography{DR-ref}

\begin{thebibliography}{10}

\bibitem{ASWK93}
B.~L. Adams, S.~I. Wright, and K.~Kunze.
\newblock Orientation imaging: {T}he emergence of a new microscopy.
\newblock {\em Journal Metallurgical and Materials Transactions A},
  24:819--831, 1993.

\bibitem{AFPA2006}
V.~Arsigny, P.~Fillard, X.~Pennec, and N.~Ayache.
\newblock Log-{E}uclidean metrics for fast and simple calculus on diffusion
  tensors.
\newblock {\em Magnetic {R}esonance in {M}edicine}, 56(2):411--421, 2006.

\bibitem{Astroem2016d}
F.~{\AA}str{\"o}m, S.~Petra, B.~Schmitzer, and C.~Schn\"orr.
\newblock {A} {G}eometric {A}pproach to {I}mage {L}abeling.
\newblock In {\em Proc.~ECCV}, 2016.

\bibitem{APSS2016}
F.~{\AA}str{\"o}m, S.~Petra, B.~Schmitzer, and C.~Schn{\"o}rr.
\newblock Image labeling by assignment.
\newblock {\em ArXiv Preprint 1603.05285}, 2016.

\bibitem{Astroem2016a}
F.~{\AA}str{\"o}m, S.~Petra, B.~Schmitzer, and C.~Schn\"orr.
\newblock {T}he {A}ssignment {M}anifold: {A} {S}mooth {M}odel for {I}mage
  {L}abeling.
\newblock In {\em Proc.~2nd Int.~Workshop on Differential Geometry in Computer
  Vision and Machine Learning}, 2016.

\bibitem{MTEX}
F.~Bachmann and R.~Hielscher.
\newblock {MTEX} -- {MATLAB} toolbox for quantitative texture analysis.
\newblock \url{http://mtex-toolbox.github.io/}, 2005--2016.

\bibitem{BHS11}
F.~Bachmann, R.~Hielscher, and H.~Schaeben.
\newblock Grain detection from 2d and 3d {EBSD} data—specification of the
  {MTEX} algorithm.
\newblock {\em Ultramicroscopy}, 111(12):1720--1733, 2011.

\bibitem{BYT2011}
E.~Bae, J.~Yuan, and X.-C. Tai.
\newblock Global minimization for continuous multiphase partitioning problems
  using a dual approach.
\newblock {\em International Journal of Computer Vision}, 92(1):112--129, 2011.

\bibitem{BN03}
M.~Belkin and P.~Niyogi.
\newblock Laplacian eigenmaps for dimensionality reduction and data
  representation.
\newblock {\em Neural Computation}, 15:1373--1396, 2003.

\bibitem{BCM2006}
A.~Buades, B.~Coll, and J.-M. Morel.
\newblock {Neighborhood filters and {PDE}s}.
\newblock {\em Numerische Mathematik}, 105:1--34, 2006.

\bibitem{BCM2010}
A.~Buades, B.~Coll, and J.-M. Morel.
\newblock Image denoising methods. {A} new nonlocal principle.
\newblock {\em SIAM Review}, 52(1):113--147, 2010.

\bibitem{BSS2016}
M.~Burger, A.~Sawatzky, and G.~Steidl.
\newblock First order algorithms in variational image processing.
\newblock In R.~Glowinski, S.~Osher, and W.~Yin, editors, {\em Operator
  Splittings and Alternating Direction Methods}. Springer, 2016.

\bibitem{CCZ2013}
X.~Cai, R.~Chan, and T.~Zeng.
\newblock {A two-stage image segmentation method using a convex variant of the
  Mumford--Shah model and thresholding}.
\newblock {\em SIAM Journal on Imaging Sciences}, 6(1):368--390, 2013.

\bibitem{CCP2012}
A.~Chambolle, D.~Cremers, and T.~Pock.
\newblock A convex approach to minimal partitions.
\newblock {\em SIAM Journal on Imaging Sciences}, 5(4):1113--1158, 2012.

\bibitem{CP11}
A.~Chambolle and T.~Pock.
\newblock A first-order primal-dual algorithm for convex problems with
  applications to imaging.
\newblock {\em Journal of Mathematical Imaging and Vision}, 40(1):120--145,
  2011.

\bibitem{CEN2006}
T.~F. Chan, S.~Esedoglu, and M.~Nikolova.
\newblock Algorithms for finding global minimizers of image segmentation and
  denoising models.
\newblock {\em SIAM Journal on Applied Mathematics}, 66(5):1632--1648, 2006.

\bibitem{CJPT11}
C.~Chaux, A.~Jezierska, J.-C. Pesquet, and H.~Talbot.
\newblock A spatial regularization approach for vector quantization.
\newblock {\em Journal of Mathematical Imaging and Vision}, 41(1-2):23--38,
  2011.

\bibitem{camino}
P.~A. Cook, Y.~Bai, S.~Nedjati-Gilani, K.~K. Seunarine, M.~G. Hall, G.~J.
  Parker, and D.~C. Alexander.
\newblock {C}amino: Open-source diffusion-{MRI} reconstruction and processing.
\newblock In {\em 14th Scientific Meeting of the International Society for
  Magnetic Resonance in Medicine}, page 2759, Seattle, WA, USA, 2006.

\bibitem{DDT2009}
C.~A. Deledalle, L.~Denis, and F.~Tupin.
\newblock Iterative weighted maximum likelihood denoising with probabilistic
  patch-based weights.
\newblock {\em IEEE Transactions on Image Processing}, 18(12):2661--2672, 2009.

\bibitem{EA06}
M.~Elad and M.~Aharon.
\newblock Image denoising via sparse and redundant representations over learned
  dictionaries.
\newblock {\em IEEE Transactions on Image Processing}, 15(12):3736--3745, 2006.

\bibitem{Fol1999}
G.~B. Folland.
\newblock {\em Real Analysis}.
\newblock Wiley \& Sons, 1999.

\bibitem{Fro1912}
G.~F. Frobenius.
\newblock {\em {\"U}ber Matrizen aus nichtnegativen Elementen}.
\newblock K{\"o}nigliche Akademie der Wissenschaften, 1912.

\bibitem{GO07}
G.~Gilboa and S.~Osher.
\newblock Nonlocal linear image regularization and supervised segmentation.
\newblock {\em SIAM Journal on Multiscale Modeling and Simulation},
  6(2):595--630, 2007.

\bibitem{graef12}
M.~Gr\"af.
\newblock A unified approach to scattered data approximation on {$\mathbb
  S^{3}$} and $\operatorname{SO}(3)$.
\newblock {\em Advances in Computational Mathematics}, 37:379--392, 2012.

\bibitem{HS2013}
S.~H\"auser and G.~Steidl.
\newblock Convex multiclass segmentation with shearlet regularization.
\newblock {\em International Journal of Computer Mathematics}, 90(1):62--81,
  2013.

\bibitem{HH1993}
L.~Herault and R.~Horaud.
\newblock Figure-ground discrimination: A combinatorial optimization approach.
\newblock {\em IEEE Transactions on Pattern Analysis and Machine Intelligence},
  15(9):899--914, 1993.

\bibitem{HB1997}
T.~Hofmann and J.~M. Buhmann.
\newblock Pairwise data clustering by deterministic annealing.
\newblock {\em IEEE Transactions on Pattern Analysis and Machine Intelligence},
  19(1):1--14, 1997.

\bibitem{HJ1986}
R.~A. Horn and C.~R. Johnson.
\newblock {\em Matrix Analysis}.
\newblock Cambridge University Press, 1986.

\bibitem{HZ83}
R.~A. Hummel and S.~W. Zucker.
\newblock On the foundation of relaxation labeling processes.
\newblock {\em IEEE Transactions on Pattern Analysis and Machine Intelligence},
  PAMI-5(3):267--287, 1983.

\bibitem{Kappes2015}
J.~H. Kappes, B.~Andres, F.~A. Hamprecht, C.~Schn{\"o}rr, S.~Nowozin, D.~Batra,
  S.~Kim, B.~X. Kausler, T.~Kr{\"o}ger, J.~Lellmann, et~al.
\newblock A comparative study of modern inference techniques for structured
  discrete energy minimization problems.
\newblock {\em International Journal of Computer Vision}, 115(2):155--184,
  2015.

\bibitem{Kol06}
V.~Kolmogorov.
\newblock Convergent tree-reweighted message passing for energy minimization.
\newblock {\em IEEE Transactions on Pattern Analysis and Machine Intelligence},
  28(10):1568--1583, 2006.

\bibitem{KZ04}
V.~Kolmogorov and R.~Zabih.
\newblock What energy functions can be minimized via graph cuts?
\newblock {\em IEEE Transactions on Pattern Analysis and Machine Intelligence},
  26(2):147–--159, 2004.

\bibitem{KuWrAdDi93}
K.~Kunze, S.~I. Wright, B.~L. Adams, and D.~J. Dingley.
\newblock Advances in automatic {EBSP} single orientation measurements.
\newblock {\em Textures and Microstructures}, 20:41--54, 1993.

\bibitem{LPS2016}
F.~Laus, J.~Persch, and G.~Steidl.
\newblock A nonlocal denoising algorithm for manifold-valued images using
  second order statistics.
\newblock {\em ArXiv Preprint 1607.08481}, 2016.

\bibitem{LLS13}
J.~Lellmann, F.~Lenzen, and C.~Schn\"orr.
\newblock Optimality bounds for a variational relaxation of the image
  partitioning problem.
\newblock {\em Journal of Mathematical Imaging and Vision}, 47(3):239–--257,
  2013.

\bibitem{LS2011}
J.~Lellmann and C.~Schn{\"o}rr.
\newblock Continuous multiclass labeling approaches and algorithms.
\newblock {\em SIAM Journal on Imaging Sciences}, 4(4):1049--1096, 2011.

\bibitem{MoBa06}
M.~Moakher and P.~G. Batchelor.
\newblock Symmetric positive-definite matrices: From geometry to applications
  and visualization.
\newblock In {\em Visualization and Processing of Tensor Fields}, pages
  285--298. Springer, Berlin, Heidelberg, 2006.

\bibitem{MS89}
D.~Mumford and J.~Shah.
\newblock Optimal approximations by piecewise smooth functions and associated
  variational problems.
\newblock {\em Communications on Pure and Applied Mathematics}, 42(5):577--685,
  1989.

\bibitem{NH16}
G.~Nolze and R.~Hielscher.
\newblock {IPF} coloring of crystal orientation data.
\newblock {\em Preprint Technische Universit\"at Chemnitz}, 2016.

\bibitem{Orland1985}
H.~Orland.
\newblock Mean-field theory for optimization problems.
\newblock {\em Journal de Physique Lettres}, 46(17):763--770, 1985.

\bibitem{Pel97}
M.~Pelillo.
\newblock The dynamics of nonlinear relaxation labeling processes.
\newblock {\em Journal of Mathematical Imaging and Vision}, 7:309--323, 1997.

\bibitem{pennec2006riemannian}
X.~Pennec, P.~Fillard, and N.~Ayache.
\newblock A {R}iemannian framework for tensor computing.
\newblock {\em International Journal of Computer Vision}, 66:41--66, 2006.

\bibitem{Per1907}
O.~Perron.
\newblock Zur {T}heorie der {M}atrizen.
\newblock {\em Mathematische Annalen}, 64(2):248--263, 1907.

\bibitem{Pey2015}
G.~Peyr\'e.
\newblock Entropic {W}asserstein gradient flows.
\newblock {\em ArXiv Preprint 1502.06216v3}, 2015.

\bibitem{RHZ1976}
A.~Rosenfeld, R.~A. Hummel, and S.~W. Zucker.
\newblock Scene labeling by relaxation operations.
\newblock {\em IEEE Transactions on Systems, Man and Cybernetics},
  SMC-6(6):420--433, 1976.

\bibitem{ROF92}
L.~I. Rudin, S.~Osher, and E.~Fatemi.
\newblock {Nonlinear total variation based noise removal algorithms}.
\newblock {\em Physica D}, 60(1):259--268, 1992.

\bibitem{Szeliski2008}
R.~Szeliski, R.~Zabih, D.~Scharstein, O.~Veksler, V.~Kolmogorov, A.~Agarwala,
  M.~Tappen, and C.~Rother.
\newblock {A comparative study of energy minimizing methods for {M}arkov random
  fields with smoothness-based priors}.
\newblock {\em IEEE Transactions on Pattern Analysis and Machine Intelligence},
  30(6):1068--1080, 2008.

\bibitem{TM98}
C.~Tomasi and R.~Manduchi.
\newblock Bilateral filtering for gray and color images.
\newblock In {\em Proc. Sixth International Conference on Computer Vision},
  pages 839--846, Bombay, India, Jan. 1998. Narosa Publishing House.

\bibitem{TPM2006}
O.~Tuzel, F.~Porikli, and P.~Meer.
\newblock Region covariance: A fast descriptor for detection and
  classification.
\newblock In {\em European {C}onference on {C}omputer {V}ision}, pages
  589--600. Springer, 2006.

\bibitem{WJ2008}
M.~J. Wainwright and M.~I. Jordan.
\newblock Graphical models, exponential families, and variational inference.
\newblock {\em Foundations and Trends in Machine Learning}, 1(1-2):1--305,
  2008.

\bibitem{We98}
J.~Weickert.
\newblock {\em Anisotropic Diffusion in Image Processing}.
\newblock Teubner, Stuttgart, 1998.

\bibitem{W1950}
H.~Wielandt.
\newblock Unzerlegbare, nicht negative matrizen.
\newblock {\em Mathematische Zeitschrift}, 52(1):642--648, 1950.

\bibitem{YFW2005}
J.~S. Yedidia, W.~T. Freeman, and Y.~Weiss.
\newblock Constructing free-energy approximations and generalized belief
  propagation algorithms.
\newblock {\em IEEE Transactions on Information Theory}, 51(7):2282--2312,
  2005.

\end{thebibliography}
\end{document}